\documentclass[10pt]{article}
\usepackage{amsmath,amssymb,amsthm}
\usepackage{color}
\usepackage[nohead,margin=1.0in]{geometry}
\usepackage{booktabs}
\usepackage{threeparttable}
\usepackage{algorithm,algorithmic}
\usepackage{verbatim}
\usepackage{url}
\usepackage{graphicx}
\usepackage{epstopdf}
\usepackage{mathtools}
\usepackage{subfigure,float}
\usepackage[colorlinks,linktocpage,linkcolor=blue]{hyperref}

\allowdisplaybreaks

\usepackage{array}
\theoremstyle{plain}
\newtheorem{theorem}{Theorem}[section]
\newtheorem{remark}{Remark}[section]

\newtheorem{lemma}{Lemma}[section]
\newtheorem{assumption}{Assumption}[section]

\newtheorem{corollary}{Corollary}[section]
\newtheorem{example}{Example}[section]

\numberwithin{equation}{section}

\renewcommand{\d}{\mathrm{d}}
\def\II{(\Omega)}
\def\tu{\tilde{u}}
\def\A{\mathcal{A}}

\title{Numerical Reconstruction of Diffusion and Potential Coefficients from Two Observations: Decoupled Recovery and Error Estimates}
\author{Siyu Cen \thanks{Department of Applied Mathematics, The Hong Kong Polytechnic University, Kowloon, Hong Kong, China. (\texttt{21037194r@connect.polyue.hk}).}
\and Zhi Zhou\thanks{Department of Applied Mathematics, The Hong Kong Polytechnic University, Kowloon, Hong Kong, China. (\texttt{zhizhou@polyu.edu.hk}).
The work of Z. Zhou is partly supported by Hong Kong Research Grants Council (15303122) and an internal grant of Hong Kong Polytechnic University (Project ID: P0031041,
Work Programme: ZZKS)}}

\begin{document}

\maketitle
\begin{abstract}
The focus of this paper is on the concurrent reconstruction of both the diffusion and potential coefficients present in an elliptic/parabolic equation, utilizing two internal measurements of the solutions. A decoupled algorithm is constructed to sequentially recover these two parameters.
In the first step, we implement a straightforward reformulation that results in a standard problem of identifying the diffusion coefficient. This coefficient is then numerically recovered, with no requirement for knowledge of the potential, by utilizing an output least-square method coupled with finite element discretization.
In the second step, the previously recovered diffusion coefficient is employed to reconstruct the potential coefficient, applying a method similar to the first step. Our approach is stimulated by a constructive conditional stability, and we provide rigorous a priori error estimates in $L^2(\Omega)$ for the recovered diffusion and potential coefficients.
Our approach is stimulated by a constructive conditional stability, and we provide rigorous a priori error estimates in $L^2(\Omega)$ for the recovered diffusion and potential coefficients. To derive these estimates, we develop a weighted energy argument and suitable positivity conditions. These estimates offer a beneficial guide for choosing regularization parameters and discretization mesh sizes, in accordance with the noise level. Some numerical experiments are presented to demonstrate the accuracy of the numerical scheme and support our theoretical results.\vskip5pt
\noindent\textbf{Key words}: identify multiple coefficients, decoupled algorithm,  conditional stability,
least-squares approach, finite element approximation,  error estimate.
\end{abstract}

\section{Introduction}\label{sec:intro}
This paper is concerned with the inverse problem for numerically recovering
spatially dependent diffusion coefficient $D(x)$ and potential coefficient $\sigma(x)$
for elliptic problems.  Let $\Omega\subset\mathbb{R}^d(d=1,2,3)$ be a convex polyhedral domain with a
boundary $\partial\Omega$.
We consider the following elliptic boundary value problem:
\begin{equation}\label{eqn:PDE-elliptic}
	\left\{\begin{aligned}
		-\nabla\cdot(D\nabla u) + \sigma u &= f,&&  \mbox{in }\Omega ,\\
		u &=  g,&&\mbox{on }\partial\Omega,
	\end{aligned}\right.
\end{equation}
where $f$ denotes a given source term and $g$ denotes the boundary data.
The solution to problem \eqref{eqn:PDE-elliptic} is denoted by $u(D,\sigma)$, to indicate its dependence on the coefficients $D$ and $\sigma$.
The inverse problem under consideration
is to recover exact coefficients $D^\dag(x)$ and $\sigma^\dag(x)$ from two interior measurements of solutions,
denoted by $u_1(D^\dag,\sigma^\dag)$ and $u_2(D^\dag,\sigma^\dag)$.
Here,  $u_i(D^{\dag},\sigma^{\dag})$ be the solution of the elliptic problem \eqref{eqn:PDE-elliptic} with source function $f_i$ and boundary data $g$.
Besides, we assume that the empirical observation $z_i^{\delta}$ is noisy with level $\delta$, i.e.,
\begin{align}\label{eq:noise_level}
	\lVert u_i(D^\dag,\sigma^\dag)-z_i^{\delta}\rVert_{L^2(\Omega)}\le \delta\quad \text{for}~~i=1,2.
\end{align}
Throughout, the diffusion coefficient and potential coefficient are respectively sought within the following admissible sets
\begin{align}\label{eq:admissible_set}
	\mathcal{A}_D=\{D\in W^{1,2}(\Omega):0< \underline{c}_D\le D\le \bar{c}_D \text{ a.e. in } \Omega\}\text{ and }\mathcal{A}_{\sigma}=\{\sigma\in L^{\infty}(\Omega):0\le \sigma\le \bar{c}_\sigma \text{ a.e. in } \Omega\}.
\end{align}
for some positive constants $\underline{c}_D, \bar{c}_D, \bar{c}_\sigma > 0$.

This inverse problem is representative among nonlinear parameter identification problems, and due to
the availability of internal measurements, high-resolution reconstructions are expected. For example,
problem \eqref{eqn:PDE-elliptic} is often used to model the behaviour of a confined inhomogenous aquifer, where $u$
represents the piezometric head, $f$ is the recharge, $D$ is hydraulic conductivity (or transmissivity
in the two-dimensional case) and $\sigma$ represents the adsorption; see \cite{Yeh:1986} for parameter
identifications in hydrology. A wide variety of coupled-physics imaging modalities in medical imaging also
leads to similar inverse problems, but the data are of slightly different form; See \cite{BalUhlmann:2013}
for several important stability results.

The case of numerically reconstructing one single parameter $D$ (with $\sigma=0$),
from the internal observation $u(x)$ in $\Omega$, 
has been extensively studied in the existing literature. See the recent work \cite{Bonito:2017} (and references therein)
on the H\"{o}lder stability of the elliptic inverse problems.
Due to the ill-posed nature of inverse problems, regularization, especially Tikhonov regularization, is customarily
employed for constructing numerical approximations (see, e.g., \cite{EnglHankeNeubauer:1996,ItoJin:2015}).
Commonly used stabilizing terms include $W^{1,2}(\Omega)$ and total variation semi-norms, which are suitable for
recovering smooth and nonsmooth diffusion coefficients, respectively.
The well-posedness and convergence (with respect to the noise level) was studied
\cite{Gutman:1990,Acar:1993,ChenZou:1999,KeungZou:1998}.
In practice, the regularized formulations are further discretized, and
the discretization necessarily introduces additional errors, which impacts the reconstruction quality.
Several studies  \cite{Gutman:1990,KeungZou:1998} have analyzed the convergence with respect to the discretization
parameter(s). See also \cite{Falk:1983,Karkkainen:1997,AlJamalGockenbach:2012,Richter:1981,KohnLowe:1988, WangZou:2010,DeckelnickHinze:2012,JinZhou:SINUM2021,ChenZhangZou:2022}
for error bounds of approximate solutions in different scenarios.

In this project, we target the simultaneous reconstruction of the diffusion coefficient $D$ and potential coefficient $\sigma$ in equation \eqref{eqn:PDE-elliptic}. This is done using two internal observations $u_1$ and $u_2$, which correspond to different source terms but share the same boundary data $g$.
It is important to highlight that a similar problem, which involves reconstructing two parameters in equation \eqref{eqn:PDE-elliptic} (with $f=0$) from two internal observations $\sigma u_1$ and $\sigma u_2$ (corresponding to distinct boundary data), has been systematically investigated in previous studies \cite{Bal:2010, BalRen:2011, BalUhlmann:2013}. This problem emerges in the context of quantitative photo-acoustic tomography in its diffusive regime.
In the aforementioned studies, the two parameters were assumed to be known at the boundary, a prerequisite for constructing a reconstruction algorithm and proving uniqueness. Additionally, the measurements were assumed to satisfy $|\nabla(u_1/u_2)| \ge \kappa>0$ almost everywhere in $\Omega$. Interestingly, these assumptions are not required in the current setting. Our investigation will address several critical aspects of the inverse problem: the conditional stability of the reconstruction, the development of an efficient numerical algorithm, and a discrete numerical scheme with a provable error estimate. One of the key challenges in this coupled problem is the appropriate selection of function spaces for the analysis, such as for conditional stability estimates, due to the diverse degree of smoothing of the forward map.
Our approach utilizes several technical tools, including the weighted stability estimate and energy technique with specific test functions \cite{Bonito:2017}. Notably,
the proposed approach differs significantly from existing ones, as the analysis naturally leads to the derivation of rigorous error bounds on the discrete approximations.

This paper also extends the argument to the parabolic equation
\begin{equation}\label{eqn:PDE-parab}
	\left\{
	\begin{aligned}
		\partial_t u-\nabla \cdot (D\nabla u)+\sigma u&=f,\quad \text{in }\Omega\times (0,T]\\
		u&=g,\quad \text{on }\partial\Omega\times (0,T]\\
		u(0)&=u_0,\quad \text{in }\Omega.
	\end{aligned}
	\right.
\end{equation}
We aim to identify the exact diffusion coefficient $D^\dag$ and the exact potential $\sigma^\dag$ from
the observation of $u(x,t)$ for $(x,t) \in (T_i - \theta, T_i] \times \Omega$ with $i=1,2$.
Here $T_1$ and $T_2$ denote two distinct time levels, and $\theta$ is a fixed small constant.
We assume that $D^\dag$ and $\sigma^{\dag}$ belong to the admissible sets $\A_D$ and $\A_\sigma$ respectively, defined in \eqref{eq:admissible_set}
and the empirical observation $z^{\delta}$ is noisy with level $\delta$, i.e.,
\begin{align}\label{eq:noise_level_para}
	\| u(D^\dag,\sigma^{\dag})-z^{\delta}\|_{L^{\infty}(T_i-\theta,T_i;L^2(\Omega))}\le \delta, \qquad i=1,2.
\end{align}
The error estimate for the numerical recovery of the single parameter has been extensively studied in different scenarios.
See e.g., \cite{JinZhou:SINUM2021,WangZou:2010} for inverse conductivity problems
and \cite{KaltenbacherRundell:2020a,JinLuQuanZhou:2023,ZhangZhangZhou:2022SINUM} for inverse potential problems.
In \cite{KaltenbacherRundell:2020a}, Kaltenbacher and Rundell analyzed the
simultaneous recovery of $\sigma^\dag$ and $D^\dag$ in one dimensional diffusion equations using the
spatial measurement $u(T)$ for two different sets of boundary conditions. The restriction on the one dimension is due to the use of the embedding $W^{1,2}\II \hookrightarrow L^\infty\II $.
In this project, we consider higher dimensional cases and use the interior observation of a single solution in $(T_i - \theta, T_i]$, $i=1,2$.
Note that the coupled nonlinear inverse problem does not admits unique recovery in general,
even for the one-dimensional case.
See a simple counterexample in the beginning of Section \ref{sec:parab}.
Therefore, such the recovery highly relies on the choice of the problem data.
We will investigate the conditional stability of the reconstruction, develop
a decoupled numerical algorithm to identify two parameters sequentially, and propose a completely discrete scheme with provable error bounds.
The argument employs some technical arguments, including decoupling the original problem into two single-parameter identification problems \cite{BalRen:2011,KaltenbacherRundell:2020a},
exploring weighted $L^2$ stability estimates by an energy argument with special test functions \cite{Bonito:2017,JinLuQuanZhou:2023}
and applying numerical analysis for the direct problems \cite{Thomee:2006}.

The rest of the paper is organized as follows. In Section \ref{sec:ellip-stab}, we show the H\"older type stability of the inverse problem for the elliptic equations,
under several positivity conditions which could be fulfilled. Then the stability estimate further motivates a decoupled recovery algorithm and the error analysis of the discrete approximation, that will be presented in Section \ref{sec:ellip-fem}. In Section \ref{sec:parab}, we extend our argument to the parabolic equation.
Numerical experiments will be presented in Section \ref{sec:numer}.  Throughout, we denote by $W^{k,p}(\Omega)$  the standard Sobolev spaces of order $k$ for any integer $k\geq0$ and real $p\geq1$, equipped with the norm $\|\cdot\|_{W^{k,p}(\Omega)}$. Moreover, we write $L^p(\Omega)$ with the norm $\|\cdot\|_{L^p(\Omega)}$ if $k=0$.
The spaces on the boundary $\partial\Omega$ are defined similarly. The notation $(\cdot,\cdot)$ denotes the standard $L^2(\Omega)$ inner product.
We denote by $c$ and $C$ generic constants not necessarily the same at each occurrence but it is always independent of the noise level, the discretization parameters and the penalty parameter.

\section{Conditional stability of inversion for elliptic equation}\label{sec:ellip-stab}
In this section, we aim to derive a conditional stability estimate for the inverse problem, which involves identifying both the diffusion and potential coefficients in an elliptic equation using two internal observations.

To accomplish this, we must first establish an assumption related to the problem data.

\begin{assumption}\label{ass:ellip}
	The source and boundary terms satisfy following properties
	\begin{itemize}
		\item[(i)] The source terms $f_1, f_2\in L^{\infty}(\Omega)$ and the boundary data $g\in W^{\frac32,2}(\partial\Omega)\cap W^{1,\infty}(\partial\Omega)$.
		\item[(ii)] The source terms $f_1, f_2\ge 0 $ and the boundary data $g \ge c_g >0$.
		\item[(iii)] The exact diffusion coefficient $D^\dag \in \A_D \cap W^{1,\infty}\II\cap W^{2,2}\II$ and the exact potential coefficient $\sigma \in \A_\sigma $.
	\end{itemize}
\end{assumption}

Under Assumption \eqref{ass:ellip}, the elliptic equation \eqref{eqn:PDE-elliptic} with source term $f_i$ ($i=1,2$) and boundary data $g$ admits a unique solution $u_i = u_i(D^\dag, \sigma^\dag)$ such that
\begin{equation}\label{eqn:reg-ellip}
	u_1, u_2 \in  W^{2,2}(\Omega)\cap W^{1,\infty}(\Omega).
\end{equation}
Moreover, by the strong maximum principle of the elliptic equation \cite[Section 6.4]{evans:book2022partial}, we conclude that there exists a constant $\underline{c}_u$ depending on $D^\dag$, $\sigma^\dag$, $g$ and $\Omega$, but independent of $f_1, f_2$ such that
\begin{equation}\label{eqn:pos-ellip}
	u_1, u_2   \ge \underline{c}_u  > 0.
\end{equation}

\subsection{Stability estimate for the recovery of diffusion coefficient}\label{subsec:stability_D}
To begin with, we eliminate the potential coefficient $\sigma$ and recover the diffusion coefficient $D$ from a reformulated elliptic problem.
Motivated by the idea in \cite{BalRen:2011}, we multiply the equation for $u_1$ by $u_2$, and the equation for $u_2$ by $u_1$,
and subtract these two relations. As a result, we eliminate the potential $\sigma$ and obtain
\begin{equation}\label{eq:elliptic_diffusion}
	\left\{
	\begin{aligned}
		-\nabla \cdot \left(Du_1^2\nabla\left( \frac{u_2}{u_1}-1\right)\right)&=f_2u_1-f_1u_2,\quad \text{in }\Omega\\
		\frac{u_2}{u_1}-1&=0,\quad \text{on }\partial\Omega
	\end{aligned}
	\right.
\end{equation}
Let $w:=\frac{u_2}{u_1}-1 $, $q:=Du_1^2 $ and $F:=f_2u_1-f_1u_2 $, hence the elliptic problem \eqref{eq:elliptic_diffusion} can be written as
\begin{equation}\label{eq:elliptic_diffusion_0_bdry}
	\left\{
	\begin{aligned}
		-\nabla \cdot \left( q\nabla w\right)&=F,\quad \text{in }\Omega\\
		w&=0,\quad \text{on }\partial\Omega
	\end{aligned}
	\right.
\end{equation}
Let the diffusion coefficient $D$, the potential coefficient $\sigma$, the source terms $f_i$ ($i=1,2$) and the boundary data $g$
satisfy Assumption \ref{ass:ellip}. Then $q=Du_1^2$ is uniformly bounded and strictly positive, and hence
we define the following admissible set
\begin{align}\label{eq:admissible_set_q}
	\mathcal{A}_q=\{q\in W^{1,2}(\Omega):0< \underline{c}_{q}\le q\le \bar{c}_{q} \text{ a.e. in } \Omega\}.
\end{align}
Note that the lower bound $\underline{c}_{q}\ge \underline{c}_D \underline{c}_u^2$ and the upper bound $\bar c_q$ depends on  $D,\sigma,f_1$ and $g$.
Meanwhile, note that $Du_1^2 \in W^{2,2}\II\cap W^{1,\infty}\II$ and $f_2u_1-f_1u_2 \in L^\infty\II$.
Consequently, we make the following assumption.
\begin{assumption}\label{ass:ellip-2}
	The exact diffusion coefficient $q^{\dag} = D^\dag|u_1(D^\dag,\sigma^\dag)|^2 \in W^{2,2}\II\cap W^{1,\infty}\II\cap \mathcal{A}_q$ and source term $F=f_2u_1(D^\dag,\sigma^\dag)-f_1u_2 (D^\dag,\sigma^\dag)\in L^{\infty}\II$.
\end{assumption}

Under Assumption \ref{ass:ellip-2}, we deduce that $w\in W_0^{1,2}\II\cap W^{2,2}\II\cap W^{1,\infty}(\Omega)$. This allows us to present the following conditional stability results in both weighted and standard $L^2(\Omega)$ norms, which play a pivotal role in the numerical analysis in Section \ref{sec:ellip-fem}.
The proof of these results modifies the proof of \cite[Theorem 2.2]{Bonito:2017}, incorporating a perturbed source term $\tilde F$.
\begin{theorem}\label{thm:stab-Du2}
Suppose that $F,\tilde F\in L^{\infty}\II$, $q,\tilde q \in \A_q$, and $q$ satisfies Assumption \ref{ass:ellip-2}.
Also, suppose the $W^{1,2}\II$-norm of $q$ and $\tilde q$ are bounded by a generic constant $C$.
Let $w$ be the solution of \eqref{eq:elliptic_diffusion_0_bdry} with the diffusion coefficient $q$ and source $F$, and $\tilde w$ as the solution with the diffusion coefficient $\tilde q$ and source $\tilde F$. Under these conditions, the following holds:
	\begin{align*}
		\int_{\Omega}\frac{(q-\tilde q)^2}{q^2}\left(q\left|\nabla w\right|^2+F w\right)\d x\le c\left(\| w-\tilde w\|_{W^{1,2}(\Omega)}+ \|F-\tilde F\|_{L^2(\Omega)}\right).
	\end{align*}
	Moreover, if  the following positive condition holds
	\begin{align}\label{eq:positive_condition}
		(q\left|\nabla w \right|^2+Fw)(x)\ge c \, \mathrm{dist}(x,\partial \Omega)^{\beta} \quad a.e. \text{ on }\Omega
	\end{align}
	for some generic constants $\beta\ge 0$ and $c>0$. Then the following estimate holds
	\begin{align}\label{eq:estimate_Du_1^2}
		\| q-\tilde q\|_{L^2(\Omega)}\le c\left(\lVert w-\tilde w\|_{W^{1,2}(\Omega)}+\lVert F-\tilde F\|_{L^2(\Omega)}\right)^{\frac{1}{2(1+\beta)}}
	\end{align}
\end{theorem}
\begin{proof}
	Define $\xi=q-\tilde q$. For any $v\in W^{1,2}_0(\Omega)$, integration by parts in \eqref{eq:elliptic_diffusion} yields
	\begin{align}\label{eq:xi_nabla_nabla_1}
		\int_{\Omega}\xi \nabla w\cdot\nabla v \d x=\int_{\Omega}\tilde q\nabla\left(\tilde w- w \right)\cdot\nabla v+(F- \tilde F)v\d x.
	\end{align}
	Besides, multiplying $\xi v/q$ on both sides of \eqref{eq:elliptic_diffusion} and applying integration by parts, we derive
	\begin{align*}
		\int_{\Omega}\frac{\xi v}{q}F \d x&=\int_{\Omega} q\nabla w\cdot\nabla\frac{\xi v}{q} =\int_{\Omega}q v\nabla w\cdot\nabla\frac{\xi}{q}+\int_{\Omega}q\frac{\xi}{q}\nabla w\cdot\nabla v,
	\end{align*}
	and hence
	\begin{align}\label{eq:xi_nabla_nabla_2}
		\int_{\Omega}\xi \nabla w\cdot\nabla v \d x=\frac{1}{2}\int_{\Omega}q \frac{\xi}{q}\nabla w\cdot\nabla v-\frac{1}{2}\int_{\Omega}q v\nabla w\cdot\nabla \frac{\xi}{q} +\frac{1}{2}\int_{\Omega}\frac{\xi v}{q}F\d x.
	\end{align}
	Now we choose the test function $v= \xi w / q $.
	Note that $q$ satisfies Assumption \ref{ass:ellip}, $\tilde q\in \mathcal{A}_q$ and $F,\tilde F\in L^{\infty}\II$. As a result, we conclude that $v \in W^{1,2}_0\II$ with
	$$  \| v \|_{L^{2}\II}  \le  \|  (q-\tilde q) w / q   \|_{L^2\II} \le \frac{2 \bar{c}_q }{\underline{c}_q} \| w \|_{L^2\II}$$
	and
	\begin{align*}
		\| \nabla v \|_{L^2\II}^2 & \le \Big\| \frac{q\nabla[(q-\tilde q) w] - (q-\tilde q) w\nabla q}{q^2} \Big\|_{L^2\II}  \\
		&\le  \frac{1}{\underline{c}_q^2}\Big( \bar{c}_q \|w \nabla (q-\tilde q) \|_{L^2} + \bar{c}_q \|(q-\tilde q) \nabla w \|_{L^2\II} + 2  \bar{c}_q \| w \|_{L^\infty\II} \| \nabla q \|_{L^2\II}  \Big) \\
		&\le \frac{1}{\underline{c}_q^2}\Big( \bar{c}_q \|w\|_{L^\infty\II} (\|\nabla q\|_{L^2\II}+ \|\nabla \tilde q\|_{L^2\II})  + 2\bar{c}_q^2 \|  \nabla w \|_{L^2\II} + 2  \bar{c}_q \| w \|_{L^\infty\II} \| \nabla q \|_{L^2\II} . \Big)
	\end{align*}
	With this test function $v$, a direct computation yields that the first two terms on the right hand side of \eqref{eq:xi_nabla_nabla_2} is equal to $\frac{1}{2}\int_{\Omega}\frac{\xi^2}{q}\left|\nabla w\right|^2\d x$. Hence, The relation \eqref{eq:xi_nabla_nabla_1} and \eqref{eq:xi_nabla_nabla_2} yields
	\begin{align*}
		\frac{1}{2}\int_{\Omega}\frac{\xi^2}{q^2}\left(q\left|\nabla w\right|^2+F w \right)\d x
		& = \int_{\Omega}\tilde q \nabla\left(\tilde w- w\right)\cdot\nabla v+\left( F-\tilde F\right)v\d x\\
		& \le c\left(\lVert w-\tilde w\|_{W^{1,2}(\Omega)}+\lVert F-\tilde F\|_{L^2(\Omega)}\right),
	\end{align*}
	With the positive condition \eqref{eq:positive_condition}, we divide the domain $\Omega$ into two parts, $\Omega_{\rho}=\{x\in \Omega:\mathrm{dist}(x,\partial\Omega)\ge \rho\}$, $\Omega_{\rho}^c=\Omega\setminus\Omega_{\rho}$.
	Thus we have
	\begin{equation*}
	\begin{aligned}
		\frac{1}{\underline{c}_q^2}  \int_{\Omega_{\rho}} |\xi|^2\d x &\le \int_{\Omega_{\rho}} \left(\frac{\xi}{q}\right)^2\d x
		\le c\rho^{-\beta}\int_{\Omega_{\rho}} \left(\frac{\xi}{q}\right)^2\rho^{\beta}\d x
		\le c\rho^{-\beta}\int_{\Omega_{\rho}} \left(\frac{\xi}{q}\right)^2 \mathrm{dist}(x,\partial\Omega)^{\beta} \d x  \\
		&\le c\rho^{-\beta}\int_{\Omega_{\rho}} \left(\frac{\xi}{q}\right)^2 \left( q\left|\nabla w \right|^2+ F w \right) \d x
		\le c\rho^{-\beta}\left(\lVert w-\tilde w\|_{W^{1,2}(\Omega)}+\lVert F-\tilde F\|_{L^2(\Omega)}\right).
	\end{aligned}
	\end{equation*}
	On the other hand,
	$\int_{\Omega_{\rho}^c} \xi^2\d x\le c|\Omega_{\rho}^c|\le c\rho.$
	Then the desired result follows by balancing the above two estimates with $\rho$.
\end{proof}
\vskip5pt

\begin{remark}\label{rmk:positive_condtion_hold}
The positive condition \eqref{eq:positive_condition} can be verified by properly selecting the source functions $f_1$ and $f_2$. For example, set $f_1\equiv 0$ and ensure $f_2\ge c_f > 0$. Given the positive lower bound of $u_1$ as presented in \eqref{eqn:pos-ellip}, we can deduce that $f_2u_1 \ge c_f \underline{c}_u > 0$.
Consequently, we obtain $F=f_2u_1-f_1u_2 = f_2u_1 \ge c_f \underline{c}_u > 0$. Hence, the positive condition \eqref{eq:positive_condition} is satisfied for $\beta=2$ according to \cite[Lemma 3.7]{Bonito:2017}.
In general, one can establish the positive condition \eqref{eq:positive_condition} with $\beta=2$ when $f_2$ is sufficiently large in comparison to $f_1$. For superior stability with $0\le \beta<2$, additional regularity assumptions are required for the domain and problem data \cite[Lemma 3.8]{Bonito:2017}.
\end{remark}\vskip5pt


Let $u_i$ ($\tilde u_i$) be the solution to the elliptic equation \eqref{eqn:PDE-elliptic} with diffusion coefficient $D$ ($\tilde D$),
potential coefficient $\sigma$ ($\tilde \sigma$), the boundary data $g$ and source functions $f_i$. Using the strict positivity  of $u_i$ and $\tilde u_i$ in \eqref{eqn:pos-ellip}
and the uniform boundedness of $u_i$ and $\tilde u_i$, we obtain
\begin{equation*}
	\begin{split}
		\| D-\tilde D \|_{L^2\II}&= \Big\| \frac{q}{u_1^2}-\frac{\tilde q}{\tilde u_1^2}\Big\|_{L^2\II}\le c \| q\tilde u_1^2-\tilde q u_1^2\|_{L^2\II}\\
		&\le c\Big(\| q\tilde u_1^2-  q u_1^2\|_{L^2\II}+\| q  u_1^2-\tilde q u_1^2\|_{L^2\II} \Big)\\
		&\le c\Big(\| \tilde u_1-   u_1\|_{L^2\II}+ \| q  -\tilde q\|_{L^2\II} \Big).
	\end{split}
\end{equation*}
Using Theorem \ref{thm:stab-Du2}, the positivity \eqref{eqn:pos-ellip} and solution regularity \eqref{eqn:reg-ellip} , we obtain
\begin{equation}\label{eqn:stab-D}
	\begin{split}
		\| D-\tilde D \|_{L^2\II}&\le c\Big(\| \tilde u_1-   u_1\|_{L^2\II}+  \big(\| w-\tilde w\|_{W^{1,2}(\Omega)}+\| F-\tilde F\|_{L^2(\Omega)}\big)^{\frac{1}{2(1+\beta)}}\Big)\\
		&\le c\Big(\| \tilde u_1-   u_1\|_{L^2\II}+  \big(   \| u_1-\tilde u_1\|_{W^{1,2}(\Omega)} +  \| u_2-\tilde u_2\|_{W^{1,2}(\Omega)} \big)^{\frac{1}{2(1+\beta)}}\Big)\\
		&\le c  \big(   \| u_1-\tilde u_1\|_{W^{1,2}(\Omega)} +  \| u_2-\tilde u_2\|_{W^{1,2}(\Omega)} \big)^{\frac{1}{2(1+\beta)}}.
	\end{split}
\end{equation}

\subsection{Stability estimate for   the recovery of potential coefficient}\label{ssec:sigma}
In the preceding section, we derived the stability for the recovery of the diffusion coefficient $D$. Now, we will shift our focus to the stability analysis for the potential coefficient $\sigma$. We introduce $\zeta :=u_2-u_1$, which satisfies the following elliptic problem:
\begin{equation}\label{eq:elliptic_potential_0_bdry}
	\left\{
	\begin{aligned}
		-\nabla \cdot \left(D \nabla \zeta\right)+\sigma \zeta&=f_2-f_1,\quad \text{in }\Omega,\\
		\zeta&=0,\quad \text{on }\partial\Omega.
	\end{aligned}
	\right.
\end{equation}
Then the next theorem provide a conditional stability for the identification of $\sigma$.

\begin{theorem}\label{thm:stability_sigma}
Assume that Assumptions \ref{ass:ellip} and \ref{ass:ellip-2} are valid. Let $\sigma, \tilde\sigma\in \mathcal{A}_{\sigma}$, with their $W^{1,2}\II$-norm being bounded by a generic constant $C$. Under these conditions, the following weighted estimate holds:
	\begin{align*}
		\lVert(\sigma-\tilde\sigma)\zeta\lVert_{L^2\II}\le c \Big(  \sum_{i=1}^2 \| u_i-\tu_i\|_{L^2\II}^2 + \sum_{i=1}^2 \| u_i-\tu_i\|_{W^{1,2}\II} +
		\| \tilde D-D\|_{L^2\II}\Big)^{\frac{1}{2}}.
	\end{align*}
	Moreover, if $f_2-f_1\ge c >0$ a.e. in $\Omega$, then for any compact subset $\Omega'\Subset \Omega$ with $\mathrm{dist}(\overline{\Omega'},\partial\Omega)>0$, there exists a positive constant $c$, depending on $\mathrm{dist}(\overline{\Omega'},\partial\Omega)$ and $D, \sigma$, such that
	\begin{align*}
		\lVert\sigma-\tilde\sigma\lVert_{L^2(\Omega')} \le  c \Big(  \sum_{i=1}^2 \| u_i-\tu_i\|_{L^2\II}^2 + \sum_{i=1}^2 \| u_i-\tu_i\|_{W^{1,2}\II} +
		\| \tilde D-D\|_{L^2\II}\Big)^{\frac{1}{2}}.
	\end{align*}
\end{theorem}
\begin{proof}
	Denote $\tilde{\zeta}$ be the solution of \eqref{eq:elliptic_potential_0_bdry} with coefficients $\tilde{D},\tilde{\sigma}$. For a test function $v\in W^{1,2}_0\II$, we consider  the $L^2\II$-inner product $\left((\sigma-\tilde\sigma)z,v\right)$. By integration by parts, we have
	\begin{align*}
		\left((\sigma-\tilde\sigma)\zeta,v\right)=&(\sigma \zeta,v)-(\tilde\sigma \zeta ,v)+(\tilde\sigma\tilde \zeta,v)-(\tilde\sigma\tilde \zeta,v)\\
		=&(\tilde D\nabla\tilde \zeta-D\nabla \zeta,\nabla v)+(\tilde\sigma(\tilde \zeta-\zeta),v).
	\end{align*}
	Now we take $v=(\sigma-\tilde\sigma)\zeta$. Recall that $\zeta\in W^{1,\infty}\II $, $\sigma,\tilde\sigma \in \A_\sigma$,
	and $\| \nabla \sigma\|_{L^2\II}, \| \nabla \tilde\sigma\|_{L^2\II}\le c$. Hence
	$$ \| \nabla v \|_{L^2\II} \le  (\|\nabla \sigma  \|_{L^2\II}+\|\nabla \tilde \sigma \|_{L^2\II})\|\zeta \|_{L^\infty\II} + 2 \bar{c}_\sigma\|\nabla \zeta\|_{L^\infty\II} \le c.$$
	As a result, noting that $D,\tilde D \in \A_D$, we obtain
	\begin{align*}
		|(\tilde D\nabla\tilde \zeta-D\nabla \zeta,\nabla v)|
		& \le \left( \lVert \tilde D(\nabla\tilde \zeta-\nabla \zeta)\|_{L^2\II}+\lVert(\tilde D-D)\nabla \zeta\|_{L^2\II}\right)\lVert \nabla v\|_{L^2\II}\\
		&\le c\Big( \bar{c}_D \|\nabla(\tilde \zeta- \zeta)\|_{L^2\II} +\|\tilde D-D\|_{L^2\II} \|\nabla \zeta \|_{L^\infty\II} \Big) \\
		&\le c\Big(  \sum_{i=1}^2 \| u_i-\tu_i\|_{W^{1,2}\II} +\lVert\tilde D-D\|_{L^2\II}\Big)
	\end{align*}
	and
	\begin{align*}
		|(\tilde\sigma(\tilde \zeta-\zeta),v)| &\le   \bar{c}_\sigma\|\tilde \zeta-\zeta\|_{L^2\II} \|v \|_{L^2\II} \le   c \sum_{i=1}^2 \| u_i-\tu_i\|_{L^2\II}^2 +  \frac12 \| v \|_{L^2\II}^2.
	\end{align*}
	Then we complete the proof of the first assertion.
	
	Since the source term in \eqref{eq:elliptic_potential_0_bdry} satisfying  $f_2 - f_1 \ge c >0$ in $\Omega$, then the  strong maximum principle \cite[Theorem 1]{Vazquez:1984} implies
	that for any $\Omega' \Subset \Omega$, there exists a positive constant $c$ depending on $\mathrm{dist}(\Omega',\partial\Omega)$ such that
	$\zeta(\sigma^\dag) \ge c > 0$ in $\Omega'$ , and consequently, the second assertion holds.
\end{proof}

\section{Finite element approximation and error analysis}\label{sec:ellip-fem}
In this section, we will introduce a numerical scheme aimed at reconstructing the diffusion coefficient $D^\dag$ and the potential coefficient $\sigma^\dag$. This is achieved using the output least-squares formulation. Taking inspiration from the stability estimate, we propose a decoupled algorithm that first recovers the diffusion coefficient $D^\dag$, followed by the reconstruction of the potential $\sigma^\dag$.

In our numerical scheme, we employ the standard Galerkin FEM \cite{BrennerScott:2008}.
Let $\mathcal{T}_h$ be a shape regular quasi-uniform triangulation of the domain $\Omega$ into $d$-simplexes,
denoted by $K$, with a mesh size $h$. Over $\mathcal{T}_h$, we define a continuous piecewise linear finite element
space $V_h$ by
\begin{align*}
	V_h=\{v_h\in W^{1,2}\II :v_h|_K \text{ is a linear function }\forall K\in \mathcal{T}_h\}
\end{align*}
and the space $V_{h}^0 = V_h\cap W^{1,2}_0\II $.
The following inverse inequality holds 
\begin{align}\label{eq:inverse_ineq}
	\lVert\varphi_h\|_{W^{1,2}\II}\le ch^{-1}\lVert \varphi_h\|_{L^2\II},\quad \forall\varphi_h\in V_h.
\end{align}
We define the $L^2$-projection $P_h:L^2\II\rightarrow V_h^0$ by
\begin{align*}
	(P_h\varphi,\chi)=(\varphi,\chi),\quad\forall\chi\in V_h^0.
\end{align*}
Note the operator $P_h$ satisfies the following error estimates \cite{Thomee:2006}: for any $s\in[1,2]$,
\begin{align}\label{eq:error_P_h}
	\lVert\varphi-P_h\varphi\|_{L^2\II}+\lVert\nabla(\varphi-P_h\varphi)\|_{L^2\II}\le h^s\lVert\varphi\|_{W^{s,2}\II},\quad\forall\varphi\in W^{s,2}\II\cap W^{1,2}_0\II.
\end{align}
Let $\mathcal{I}_h$ be the Lagrange interpolation operator associated with the finite element space $V_h$. It satisfies the following error estimate for $s=1,2$ and $1\le p\le \infty$ {(with $sp>d$ if $p>1$ and $sp\ge d$ if $p=1$)} \cite{BrennerScott:2008}:
\begin{align}\label{eq:error_I_h}
	\lVert\varphi-\mathcal{I}_h\varphi\|_{L^p\II}+\lVert\nabla(\varphi-\mathcal{I}_h\varphi)\|_{L^p\II}\le h^s\lVert\varphi\|_{W^{s,p}\II},\quad\forall\varphi\in W^{s,p}\II.
\end{align}

\subsection{Step one: numerically recover diffusion coefficient $D^\dag$}\label{subsec:recover_D}
Recall that the elliptic problem \eqref{eq:elliptic_diffusion_0_bdry} enables to recover the diffusion coefficient without the knowledge of potential.
Note that the exact solution $u_i^\dag:=u_i(D^\dag,\sigma^\dag)$, with $i=1,2$, is strictly positive with a fixed lower bound
(cf. \eqref{eqn:pos-ellip}). For ease of simplicity, we  assume  that the empirical observation $z_i^{\delta}$ satisfies the same positive lower bound.
We define
\begin{equation*}
	w^\delta(x) =\frac{z_2^{\delta}(x)}{z_1^{\delta}(x)}-1\quad \text{and}\quad  w^\dag(x) =\frac{u_2^{\dag}(x)}{u_1^{\dag}(x)}-1.
\end{equation*}
Using \eqref{eq:noise_level}, \eqref{eqn:reg-ellip} and \eqref{eqn:pos-ellip}, we derive
\begin{equation*}
	\begin{split}
		\| w^\delta - w^\dag\|_{L^2\II}  &= \Big\| \frac{z_1^{\delta}u_2^\dag-z_2^{\delta}u_1^\dag}{u_1^\dag z_1^{\delta}} \Big\|_{L^2\II}
		\le \Big\| \frac{z_1^{\delta}u_2^\dag-u_1^\dag u_2^\dag}{u_1^\dag z_1^{\delta}} \Big\|_{L^2\II}
		+ \Big\| \frac{u_1^{\dag} u_2^\dag-z_2^{\delta}u_1^\dag}{u_1^\dag z_1^{\delta}} \Big\|_{L^2\II} \\
		&\le \frac{1}{\underline{c}_u^2} \Big( \| z_1^{\delta}u_2^\dag-u_1^\dag u_2^\dag \|_{L^2\II} +  \|u_1^\dag u_2^\dag  - z_2^{\delta}u_1^\dag \|_{L^2\II}  \Big) \le c\delta.
	\end{split}
\end{equation*}
Moreover, we should also take care of the source term $F$ in \eqref{eq:elliptic_diffusion_0_bdry}
where the exact solutions $u_1$ and $u_2$ should be replaced with noisy observations.
Hence we define
$F^{\delta}:=f_2z_1^{\delta}-f_1z_2^{\delta}$ with
\begin{equation}\label{eqn:Fd}
	\| F-F^{\delta}\|_{L^2\II}\le \|(u_1-z_1^{\delta})f_2\|_{L^2\II}
	+\|(u_2-z_2^{\delta})f_1\|_{L^2\II}\le c\delta.
\end{equation}

We look for the numerical reconstruction of diffusion coefficient of the system \eqref{eq:elliptic_diffusion_0_bdry} in the admissible set $\mathcal{A}_{q,h}=\mathcal{A}_q\cap V_h$.
The finite element scheme reads
\begin{align}\label{eq:min_D_dis}
	\min_{q_h\in \mathcal{A}_{q,h}} J_{\alpha_1}(q_h)=\frac{1}{2}\lVert w_h(q_h)-w^{\delta}\|_{L^2\II}^2+\frac{\alpha_1}{2}\lVert \nabla q_h\|_{L^2\II}^2
\end{align}
where $w_h(q_h)\in V_h^0$ is a weak solution of
\begin{equation}\label{eq:min_D_dis_restriction}
	(q_h\nabla w_h,\nabla v_h)=(F^{\delta},v_h) \quad\forall \,v_h\in V_{h}^0.
\end{equation}
For any $\gamma,h>0$, there exists at least one minimizer  $q_h^*$ to problem \eqref{eq:min_D_dis}-\eqref{eq:min_D_dis_restriction}; see \cite{HinzeKaltenbacher:2018,Zou:1998} for related analysis
for the well-posedness and convergence.
Then our objective is to bound the error $q^{\dag}-q_h^*$, where $q^{\dag}=D^{\dag}u_1^2$ is the exact coefficient satisfying
$q^{\dag}\in W^{2,2}(\Omega)\cap W^{1,\infty}(\Omega)$ provided that Assumption \ref{ass:ellip-2} holds valid.
To this end, we state the following \textsl{a priori} estimate for $\| w_h(q_h^*)-w(q^{\dag}) \|_{L^2\II}$ and $  \|\nabla q_h^{*}\|_{L^2\II} $.

\begin{lemma}\label{lem:ellip-est-1}
Suppose that $q^\dag$ and $F$ satisfies Assumption \ref{ass:ellip-2}. Let $q_h^* \in \mathcal{A}_{q,h}$
be a minimizer of problem \eqref{eq:min_D_dis}-\eqref{eq:min_D_dis_restriction}. Then there holds
\begin{equation*}
	\begin{aligned}
 \| w_h(q_h^*)-w^{\delta}\|_{L^2\II}^2+\alpha_1 \|\nabla q_h^{*}\|_{L^2\II}^2  \le c (h^4 +\delta^2+\alpha_1),
	\end{aligned}
\end{equation*}
\end{lemma}

\begin{proof}
First of all, we notice the following estimate
\begin{align}\label{eq:w_h(I_hq)-w(q)}
	\| w_h(\mathcal{I}_hq^{\dag})-w(q^{\dag}) \|_{L^2\II}\le c(h^2+\delta).
\end{align}
The proof follows from the Lax--Milgram lemma and the standard duality argument, similar to that of \cite[Lemma A.1]{JinZhou:SINUM2021}.
The only difference is that the source term $F^\delta$ in \eqref{eq:min_D_dis_restriction} is noisy with level $\delta$, cf. \eqref{eqn:Fd}.
Since $q_h^*$ is a minimizer of $J_{\alpha_1,h}$, we have $J_{\alpha_1,h}(q_h^*)\le J_{\alpha_1,h}(\mathcal{I}_h q^{\dag})$. This combined with the estimate \eqref{eq:w_h(I_hq)-w(q)}
leads to
\begin{equation*}
	\begin{aligned}
		\| w_h(q_h^*)-w^{\delta}\|_{L^2\II}^2+\alpha_1 \|\nabla q_h^{*}\|_{L^2\II}^2
		\le&   \| w_h(\mathcal{I}_h q^{\dag})-w^{\delta}\|_{L^2\II}^2+\alpha_1 \|\nabla\mathcal{I}_h q^{\dag}\|_{L^2\II}^2\\
		\le& \| w_h(\mathcal{I}_h q^{\dag})-w(q^{\dag})\|_{L^2\II}^2+ \| w(q^{\dag})-w^{\delta}\|_{L^2\II}^2+\alpha_1\|\nabla\mathcal{I}_h q^{\dag}\|_{L^2\II}^2\\
		\le&  c (h^4 +\delta^2+\alpha_1),
	\end{aligned}
\end{equation*}
Then the proof is complete.
\end{proof}\vskip5pt

The following theorem establishes a bound for the error $q^{\dag}-q_h^*$. The approach is inspired by the stability estimate provided in Theorem \ref{thm:stab-Du2}.
\vskip5pt

\begin{theorem}\label{thm:error_weighted_diffusion}
	Suppose Assumptions \ref{ass:ellip} and \ref{ass:ellip-2} hold valid. Let $q^\dag = D^\dag |u_1(D^\dag,\sigma^\dag)|^2 \in \mathcal{A}_{q}$ be the exact parameter in \eqref{eq:elliptic_diffusion_0_bdry}, $w^\dag=w(q^{\dag})$ be the solution of \eqref{eq:elliptic_diffusion_0_bdry}, and $q_h^*\in \mathcal{A}_{q,h}$ be a minimizer of problem \eqref{eq:min_D_dis}-\eqref{eq:min_D_dis_restriction}. Then with $\eta=h^2+\delta+\alpha_1^{\frac{1}{2}}$, there holds
	\begin{align*}
		\int_{\Omega}\left(\frac{q^{\dag}-q_h^*}{q^{\dag}}\right)^2\left( q^{\dag}|\nabla w^\dag|^2+Fw(q^{\dag})\right)\d x\le
		c \big( (h\eta\alpha_1^{-\frac{1}{2}}+\min(h+h^{-1}\eta,1))\eta\alpha_1^{-\frac{1}{2}}+\delta\big).
	\end{align*}
	Moreover, let $D_h^* = q_h^*/|z_1^\delta|^2$. If the positive condition \eqref{eq:positive_condition} holds with some $\beta\ge0$, then
	\begin{align*}
		\| D_h^* -  D^\dag \|_{L^2(\Omega)} \le c \big( (h\eta\alpha_1^{-\frac{1}{2}}+\min(h+h^{-1}\eta,1))\eta\alpha_1^{-\frac{1}{2}}+\delta\big)^{\frac{1}{2(1+\beta)}}.
	\end{align*}
\end{theorem}
\begin{proof}
	For any test function $\varphi\in W^{1,2}_0\II$, the  the weak formulation of $w(q^{\dag})$ and $w_h(q_h^*)$ imply
	\begin{align*}
		\left((q^{\dag}-q_h^*)\nabla w^\dag,\nabla\varphi\right)=&\left((q^{\dag}-q_h^*)\nabla w^\dag,\nabla(\varphi-P_h\varphi)\right)+\left((q^{\dag}-q_h^*)\nabla w^\dag,\nabla P_h\varphi\right)\\
		=&-\left(\nabla\cdot\left((q^{\dag}-q_h^*)\nabla w^\dag\right),\varphi-P_h\varphi\right)\\
		&+\left(q_h^*(\nabla w_h(q_h^*)-\nabla w^\dag),\nabla P_h\varphi\right)+\left( F-F^{\delta},P_h\varphi\right).
	\end{align*}
	Motivated by the proof of Theorem \ref{thm:stab-Du2}, we choose $\varphi=\frac{q^{\dag}-q_h^*}{q^{\dag}}w^\dag$. A direct computation leads to
	\begin{align*}
		\nabla\varphi=\left(\frac{q^{\dag}\nabla(q^{\dag}-q_h^*)-(q^{\dag}-q_h^*)\nabla q^{\dag}}{q^{\dag2}}\right)w+\frac{q^{\dag}-q_h^*}{q^{\dag}}\nabla w^\dag.
	\end{align*}
	By the box constraint of the admissible sets $\mathcal{A}_q$ and $\mathcal{A}_{q,h}$ as well as the regularity of $w^\dag$, we derive
	\begin{align*}
		\|\nabla\varphi\|_{L^2\II}&\le c  \| q^{\dag}\nabla(q^{\dag}-q_h^*)-(q^{\dag}-q_h^*)\nabla q^{\dag} \|_{L^2\II} \|w^\dag\|_{L^\infty\II}+  c \| q^{\dag}-q_h^* \|_{L^\infty\II} \|\nabla w^\dag\|_{L^2\II}\\
		&\le c \| q^{\dag} \|_{L^\infty\II} (\|\nabla q^{\dag}\|_{L^2\II} + \| \nabla q_h^*\|_{L^2\II}) + c ( \| q^{\dag}\|_{L^\infty\II} + \| q_h^* \|_{L^\infty\II}) \|\nabla q^{\dag} \|_{L^2\II}) + c \\
		&\le c (1+ \| \nabla q_h^* \|_{L^2\II}).
	\end{align*}
	Next, according to the box constraint of $q^\dag$ and $q_h^*$, the regularity of $q^\dag$ and $w^\dag$, the approximation property of $P_h$ in \eqref{eq:error_P_h},
	as well as Lemma \ref{lem:ellip-est-1}, we have
	\begin{align*}
		&\quad |\left(\nabla\cdot\left((q^{\dag}-q_h^*)\nabla w^\dag\right),\varphi-P_h\varphi\right)|\\
		&\le \left(\|  q^{\dag}\|_{L^{\infty}\II}\| \nabla w^{\dag}\|_{L^2\II}+\|  q^{\dag}-q_h^*\|_{L^{\infty}\II}\|  \Delta w^{\dag}\|_{L^2\II}+\| \nabla q_h^*\|_{L^2\II}\| \nabla w^{\dag}\|_{L^{\infty}\II}\right)\| \varphi-P_h\varphi\|_{L^2\II}\\
		&\le c \left( 1 + \| \nabla q_h^*\|_{L^2\II} \right) \| \varphi-P_h\varphi\|_{L^2\II}
		\le ch\left(1+\| \nabla q_h^*\|_{L^2\II} \right)\| \nabla \varphi\|_{L^2\II}\\
		&\le  ch\left(1+\| \nabla q_h^*\|_{L^2\II} \right)^2 \le c h(1+\alpha_1^{-1}\eta^2).
	\end{align*}
	For the remaining terms, by the triangle inequality, the inverse inequality \eqref{eq:inverse_ineq}, the stability and approximation of $P_h$, and
	 Lemma \ref{lem:ellip-est-1}, we have
	\begin{align*}
		\| \nabla(w_h(q_h^*)-w^\dag )\|_{L^2\II}&\le \| \nabla(w_h(q_h^*)-P_hw^\dag )\|_{L^2\II}+\| \nabla(P_hw^\dag -w^\dag )\|_{L^2\II}\\
		&\le c h^{-1}\|  w_h(q_h^*)-P_hw^\dag\|_{L^2\II} + c h\|  w^\dag \|_{W^{2,2}\II} \\
		&\le ch^{-1}\left(\|  w_h(q_h^*)-P_hw_h(q_h^*)\|_{L^2\II}+\|  P_hw_h(q_h^*)-P_h w^\dag\|_{L^2\II} \right) + c h\|  w^\dag \|_{W^{2,2}\II} \\
		&\le c\left(h+h^{-1}\|  w_h(q_h^*)-w(q^{\dag})\|_{L^2\II}  \right) \le c \left(h+h^{-1}\eta \right).
	\end{align*}
	Meanwhile, the Lax--Milgram lemma implies $\| \nabla(w_h(q_h^*)-w^\dag)\|_{L^2\II} \le \| \nabla w_h(q_h^*)\|_{L^2\II} + \| \nabla w^\dag\|_{L^2\II} \le C $.
	As a result, we derive
	\begin{align*}
		| (q_h^*(\nabla w_h(q_h^*)-\nabla w^\dag),\nabla P_h\varphi)+( F-F^{\delta},P_h\varphi)|
		&\le c \| \nabla(w_h(q_h^*)-w^\dag)\|_{L^2\II}\| \nabla\varphi\|_{L^2\II}+\delta\| \varphi\|_{L^2\II}\\
		&\le  c\left(\min(h+h^{-1}\eta,1)\alpha_1^{-1/2}\eta+\delta\right)
	\end{align*}
	Then using integration by parts and $F=-\nabla\cdot(q^{\dag}\nabla w(q^{\dag}))$, we have
	\begin{align*}
		\left((q^{\dag}-q_h^*)\nabla w(q^{\dag}),\nabla\varphi\right) =\frac{1}{2}\int_{\Omega}\left(\frac{q^{\dag}-q_h^*}{q^{\dag}}\right)^2\left( q^{\dag}|\nabla w(q^{\dag})|^2+Fw(q^{\dag})\right)\d x.
	\end{align*}
	Consequently, under positivity condition \eqref{eq:positive_condition}, the same argument of \eqref{eq:estimate_Du_1^2} yields
	\begin{align*}
		\| q_h^* -  q^\dag \|_{L^2(\Omega)} \le c \big( (h\eta\alpha_1^{-\frac{1}{2}}+\min(h+h^{-1}\eta,1))\eta\alpha_1^{-\frac{1}{2}}+\delta\big)^{\frac{1}{2(1+\beta)}}.
	\end{align*}
	Finally, we let $D_h^* = q_h^*/|z_1^\delta|^2$ and recall that $D^\dag = q^\dag / |u_1^\dag|^2$ with $u_1^\dag = u_1 (D^\dag, \sigma^\dag)$.
	Using the box constraint of $q_h^*$, the estimate \eqref{eq:noise_level}, the regularity of $u_1^\dag$ (provided that Assumption \ref{ass:ellip} holds valid),
	and  the fact that $u_1^\dag, z_1^\delta \ge c_u$, we derive
	\begin{align*}
		\|D_h^*  - D^\dag \|_{L^2\II} & \le c (\| q_h^*\|_{L^\infty\II} \| (z_1^\delta)^2 - (u_1^\dag)^2 \|_{L^2\II} + \| q_h^* - q^\dag \|_{L^2\II} \| u_1^\dag \|_{L^\infty\II}) \\
		&\le c (\delta + \| q_h^* - q^\dag \|_{L^2\II}) \le c \big( (h\eta\alpha_1^{-\frac{1}{2}}+\min(h+h^{-1}\eta,1))\eta\alpha_1^{-\frac{1}{2}}+\delta\big)^{\frac{1}{2(1+\beta)}}.
	\end{align*}
	This completes the proof of the theorem.
\end{proof}

\begin{remark}\label{rem:error_diffusion_D}
Theorem \ref{thm:error_weighted_diffusion} serves as a guideline for the \textsl{a priori} selection of algorithmic parameters,
suggesting $\alpha_1 \sim \delta^2$ and $h \sim \sqrt\delta$. Given the positivity condition \eqref{eq:positive_condition} with $\beta\ge0$,
the following estimate is valid:
\begin{align*}
\| D^{\dag}-D_h^*\|_{L^2\II}\le c \delta^{\frac{1}{4(1+\beta)}}.
\end{align*}
This estimate is optimal with respect to the conditional stability estimate \eqref{eqn:stab-D}.
\end{remark}

\subsection{Step two: Numerically recover the potential $\sigma$}\label{subsec:recover_sigma}

In this section, we give the reconstruction formula of potential $\sigma$ and corresponding error analysis. We make following regularity assumption for true potential $\sigma^{\dag}$.
\begin{assumption}\label{assume:potential_regularity}
	The potential  $\sigma^{\dag}$ satisfies $\sigma^{\dag}\in  W^{2,2}\II \cap \mathcal{A}_{\sigma}$.
\end{assumption}

Recall $\zeta  =u_2-u_1$  satisfies following elliptic equation
\begin{equation*}
	\left\{
	\begin{aligned}
		-\nabla \cdot \left(D^\dag \nabla \zeta\right)+\sigma^\dag \zeta&=f_2-f_1,\quad \text{in }\Omega,\\
		\zeta&=0,\quad \text{on }\partial\Omega.
	\end{aligned}
	\right.
\end{equation*}
The noisy observational data for equation \eqref{eq:elliptic_potential_0_bdry} is $\zeta^{\delta}=z_2^{\delta}-z_1^{\delta}$ and $\|  \zeta-\zeta^{\delta}\|_{L^2\II}\le c\delta$.

From now on, let $H$ represent the spatial mesh size, which may differ from $h$ used in the previous section.
Utilizing the output least-squares formulation, we can approximate the recovery of the potential as follows:
\begin{align}\label{eq:min_sigma_dis}
	\min_{\sigma_H\in \mathcal{A}_{\sigma,H}} \mathcal{J}_{\alpha_2,H }(\sigma_H)=\frac{1}{2}\|  \zeta_H(\sigma_H)-\zeta^{\delta}\|_{L^2\II}^2+\frac{\alpha_2}{2}\|  \nabla \sigma_H\|_{L^2\II}^2
\end{align}
where $\mathcal{A}_{\sigma,H}=\mathcal{A}_{\sigma}\cap V_H$ and $\zeta_H(\sigma_H)\in V_H^0$ is the solution to the finite dimensional problem
\begin{equation}\label{eq:min_sigma_dis_restriction}
	(D_h^*\nabla \zeta_H,\nabla v_H)+(\sigma_H \zeta_H,v_H)=(f_2-f_1,v_H)\qquad\forall ~v_H\in V_{H}^0.
\end{equation}
Here $D_h^*$ is the diffusion coefficient we reconstructed in Theorem \ref{thm:error_weighted_diffusion}.
As stated in Remark \ref{rem:error_diffusion_D},
we have the \textsl{a priori} error estimate
$$\|  D_h^*-D^{\dag}\|_{L^2\II}\le c\, \delta^{\gamma} =:\epsilon \qquad  \text{with} ~~\gamma=(4(1+\beta))^{-1}.$$
The discrete problem \eqref{eq:min_sigma_dis}-\eqref{eq:min_sigma_dis_restriction} is well-posed: there exists at least one global minimizer $\sigma_H^*$ to problem \eqref{eq:min_sigma_dis}-\eqref{eq:min_sigma_dis_restriction}
and it depends continuously on the data perturbation. Then we aim to establish an \textsl{a priori} error bound between
$\sigma_H^*$ and $\sigma^{\dag}$.

To accomplish this, we initially derive a bound for $\zeta_H(\mathcal{I}_H \sigma^{\dag })-\zeta(\sigma^{\dag})$. This is achieved using the standard estimates applicable to the finite element method.
\begin{lemma}\label{lem:zhIhs-zs}
	Suppose Assumption \ref{assume:potential_regularity} holds, \textcolor{black}{$D^{\dag}\in \mathcal{A}_D\cap W^{1,\infty}\II$} and $\|  D_h^*-D^{\dag}\|_{L^2\II}\le \epsilon$.
	Let $\zeta(\sigma^\dag)$ be the solution to the elliptic problem \eqref{eq:elliptic_potential_0_bdry}, while $\zeta_H(\mathcal{I}_H \sigma^{\dag})$ be the solution to the finite dimensional problem \eqref {eq:min_sigma_dis_restriction} where $\sigma_H$ is replaced with $\mathcal{I}_H \sigma^{\dag }$. Then
	\begin{align*}
		\|  \zeta_H(\mathcal{I}_H \sigma^{\dag })-\zeta(\sigma^{\dag})\|_{L^2\II}\le c (H^2+\epsilon).
	\end{align*}
\end{lemma}
\begin{proof}
We apply the following splitting
$$\zeta_H(\mathcal{I}_H \sigma^{\dag })-   \zeta(\sigma^{\dag})  =(\zeta_H(\mathcal{I}_H \sigma^{\dag })-\overline{\zeta}_H )
	+ ( \overline{\zeta}_H - \tilde{\zeta}_H )+ (   \tilde{\zeta}_H -   \zeta(\sigma^{\dag})) =: \sum_{j=1}^3 e_j,$$
where $\overline{\zeta}_H $ and  $\tilde{\zeta}_H$ respectively satisfy
\begin{align*}
		(D^{\dag}\nabla \overline{\zeta}_H ,\nabla v_H)+(\mathcal{I}_H\sigma^\dag \overline{\zeta}_H ,v_H)&=(f_2-f_1,v_H)\qquad\forall ~v_H\in V_H^0,\quad \text{and} \\
		(D^{\dag}\nabla \tilde{\zeta}_H ,\nabla v_H)+(\sigma^\dag \tilde{\zeta}_H ,v_H)&=(f_2-f_1,v_H)\qquad\forall ~v_H\in V_{H}^0.
\end{align*}
We begin with the $L^2\II$ bound of $e_1$, which satisfies
	\begin{align*}
		(D_h^*\nabla e_1 ,\nabla v_H)+((\mathcal{I}_H\sigma^\dag) e_1,v_H)&=((D^\dag-D_h^*) \nabla\overline\zeta_H,\nabla v_H)\qquad\forall ~v_H\in V_H^0.
	\end{align*}
	Now we choose $v_H = e_1$ in the above relation.
	By the regularity $\| D^\dag \|_{W^{1,\infty}\II} + \| (\mathcal{I}_H) \sigma^\dag \|_{L^\infty\II} \le c$, we conclude that
	 $\| \nabla\overline\zeta_H \|_{L^\infty\II}\le c$ (cf. \cite{Rannacher:1982} and \cite[Theorem 2]{Guzman:2009}).
	This together with the fact that $D_h^{\dag} \in \mathcal{A}_D$ and Poincar{\'e}'s inequality  implies
	\begin{align*}
		\| e_1 \|_{H^1\II} \le \| D^\dag-D_h^*\|_{L^2\II}  \| \nabla\overline\zeta_H\|_{L^\infty\II} \le c \epsilon.
	\end{align*}
	Now we turn to the second term $e_2$ which satisfies
	\begin{align*}
		(D^\dag \nabla e_2 ,\nabla v_H)+((\mathcal{I}_H \sigma^\dag) e_2,v_H)&= ((\sigma-\mathcal{I}_H \sigma^\dag) \tilde{\zeta}_H,v_H)\qquad\forall ~v_H\in V_H^0.
	\end{align*}
	Letting $v_H=e_2$ and using the fact that $\| \tilde{\zeta}_H  \|_{L^\infty\II} \le C$, we arrive at
	\begin{align*}
		\| \nabla e_2 \|_{L^2\II} \le c \| \sigma-\mathcal{I}_H \sigma^\dag\|_{L^2}  \|\tilde{\zeta}_H \|_{L^\infty\II} \le c H^2,
	\end{align*}
	where we use the estimate \eqref{eq:error_I_h} and the assumption that $\sigma^\dag \in W^{2,2}\II$.
	Finally, the estimate for the third term,  $\|  e_3 \|_{L^2\II}\le CH^2$, can be estimated directly by applying the standard argument.
\end{proof}\vskip5pt

The subsequent lemma offers useful bounds for the state $\zeta_H(\sigma_H^*)$ and the $W^{1,2}(\Omega)$ seminorm of $\sigma_H^*$.
\vskip5pt
\begin{lemma}\label{lem:zhsh*-zs}
	Suppose Assumption \ref{assume:potential_regularity} holds, \textcolor{black}{$D^{\dag}\in \mathcal{A}_D\cap W^{1,\infty}\II$ } and $\|  D_h^*-D^{\dag}\|_{L^2\II}\le \epsilon$. Then there is
	\begin{align*}
		\|  \zeta(\sigma^{\dag})-\zeta_H(\sigma_H^*)\|_{L^2\II}+ \alpha_2^{\frac{1}{2}}\|  \nabla \sigma_H^*\|_{L^2\II}\le c(H^2+\epsilon+\delta +\alpha_2^{\frac{1}{2}})
	\end{align*}
\end{lemma}
\begin{proof}
	Since $\sigma_H^*$ is a minimizer of $\mathcal{J}_{\alpha_2,H}$, we have $\mathcal{J}_{\alpha_2,H}(\sigma_H^*)\le \mathcal{J}_{\alpha_2,H}(\mathcal{I}_H \sigma^{\dag}) $.
	Then we derive
	\begin{align*}
		&\|  \zeta_H(\sigma_H^*)-\zeta^{\delta}\|_{L^2\II}^2+\alpha_2\| \nabla \sigma_H^{*}\|_{L^2\II}^2\\
		\le& \|  \zeta_H(\mathcal{I}_H \sigma^{\dag})-\zeta^{\delta}\|_{L^2\II}^2+\alpha_2\| \nabla\mathcal{I}_H \sigma^{\dag}\|_{L^2\II}^2\\
		\le& \|  \zeta_H(\mathcal{I}_H \sigma^{\dag})-\zeta(\sigma^{\dag})\|_{L^2\II}^2+\|  \zeta(\sigma^{\dag})-\zeta^{\delta}\|_{L^2\II}^2+\alpha_2\| \nabla\mathcal{I}_H \sigma^{\dag}\|_{L^2\II}^2\\
		\le & c \left( (H^2+\epsilon)^2 +\delta^2+\alpha_2\right),
	\end{align*}
	where in the last inequality we use the result in Lemma \ref{lem:zhIhs-zs}.
\end{proof}

Then we are ready to show the error bound of numerically recovered potential.

\begin{theorem}\label{thm:error_weighted_potential}
	Suppose that Assumption \ref{assume:potential_regularity} holds valid, $D^{\dag}\in \mathcal{A}_D\cap W^{1,\infty}\II$ and $\|  D_h^*-D^{\dag}\|_{L^2\II}\le \epsilon$.
	Let $\sigma_H^*$ be the numerical reconstruction of the potential given in  \eqref{eq:min_sigma_dis}-\eqref{eq:min_sigma_dis_restriction}.
	Then with $\eta=H^2+\epsilon+\delta+\alpha_2^{\frac{1}{2}}$, there holds
	\begin{align*}
		\|  (\sigma^{\dag}-\sigma_H^*)\zeta(\sigma^{\dag})\|_{L^2\II}
\le   c \big(H\alpha_2^{-\frac{1}{2}}\eta+\eta+ \big(\alpha_2^{-\frac{1}{2}}\eta(\min\{H+H^{-1}\eta,1\}+\epsilon )\big)^{\frac{1}{2}} \big).
	\end{align*}
	Finally, if $f_2-f_1\ge c>0$ a.e. in $\Omega$, then for any $\Omega' \Subset \Omega$, there exists a constant $c$ depending on $\mathrm{dist}(\Omega',\partial\Omega)$ and  $D^{\dag},\sigma^{\dag}$ such that
	\begin{equation*}
		\| \sigma^{\dag}-\sigma_H^*\|_{L^2(\Omega')}
\le  c \big(H\alpha_2^{-\frac{1}{2}}\eta+\eta+ \big(\alpha_2^{-\frac{1}{2}}\eta(\min\{H+H^{-1}\eta,1\}+\epsilon )\big)^{\frac{1}{2}} \big).
	\end{equation*}
\end{theorem}
\begin{proof}
	For any test function $\varphi\in W^{1,2}_0\II$, by weak formulation of $\zeta_H(\sigma_H^*)$ and $\zeta(\sigma^{\dag})$, we have
	\begin{align*}
		&\left( (\sigma^{\dag}-\sigma_H^*)\zeta(\sigma^{\dag}),\varphi \right)\\
		=&\left( (\sigma^{\dag}-\sigma_H^*)\zeta(\sigma^{\dag}),\varphi-P_H\varphi \right)+\left( (\sigma^{\dag}-\sigma_H^*)\zeta(\sigma^{\dag}),P_H\varphi \right) \\
		=&\left( (\sigma^{\dag}-\sigma_H^*)\zeta(\sigma^{\dag}),\varphi-P_H\varphi \right) +\left(D_h^*\nabla \zeta_H(\sigma_H^*)-D^{\dag}\nabla \zeta(\sigma^{\dag}),\nabla P_H\varphi \right)+\left(\sigma_H^*(\zeta_H(\sigma_H^*)-\zeta(\sigma^{\dag})),P_H\varphi \right)\\
		=:& \sum_{j=1}^n \mathrm{I}_j.
	\end{align*}
	Now, we take $\varphi=(\sigma^{\dag}-\sigma_H^*)\zeta(\sigma^{\dag})$. A direct computation implies $\varphi\in W^{1,2}_0\II$ and
	$$\| \varphi\|_{L^2\II}\le c\quad \text{and} \quad \|  \nabla\varphi\|_{L^2\II}\le c(1+\| \nabla\sigma_H^*\|_{L^2\II}).$$
	By the box constraint of $\A_{\sigma,H}$, the error estimate of $P_H$ in \eqref{eq:error_P_h} and Lemma \ref{lem:zhsh*-zs},
	we have
	\begin{equation}\label{eqn:est-sig-1}
		\begin{aligned}
			|\mathrm{I}_1| \le& \|  (\sigma^{\dag}-\sigma_H^*)\zeta(\sigma^{\dag})\|_{L^{2}\II} \|  \varphi-P_H\varphi\|_{L^2\II}
			\le cH\| (\sigma^{\dag}-\sigma_H^*)\zeta(\sigma^{\dag})\|_{L^{2}\II}\| \nabla\varphi\|_{L^2\II}\\
			\le &cH\| (\sigma^{\dag}-\sigma_H^*)\zeta(\sigma^{\dag})\|_{L^{2}\II}(1+\| \nabla\sigma_h^*\|_{L^2\II})
			\le cH^2\alpha_2^{-1}\eta^2+\frac{1}{3}\| (\sigma^{\dag}-\sigma_H^*)\zeta(\sigma^{\dag})\|_{L^{2}\II}^2.
		\end{aligned}
	\end{equation}
	By the stability of $P_H$ and Lemma \ref{lem:zhsh*-zs} we conclude
\begin{equation}\label{eqn:est-sig-2}
\begin{aligned}
	|\mathrm{I}_3| \le&\|  \sigma_H^* \|_{L^{\infty}\II}\|  \zeta_H(\sigma_H^*)-\zeta(\sigma^{\dag}) \|_{L^2\II}\|  P_H\varphi \|_{L^2\II}\\
	 		\le& c\|  \zeta_H(\sigma_h^*)-\zeta(\sigma^{\dag}) \|_{L^2\II}^2+\frac{1}{3}\|  \varphi\|_{L^{2}\II}^2\le c\eta^2+\frac{1}{3}\|  \varphi\|_{L^{2}\II}^2.
\end{aligned}
\end{equation}
	Finally, we turn to the term $\mathrm{I}_2$ and use the splitting
	$$ \mathrm{I}_{2}  =  \left(D_h^*\nabla \zeta_H(\sigma_H^*)-D_h^*\nabla \zeta(\sigma^{\dag}),\nabla P_H\varphi \right) + \left(D_h^*\nabla \zeta(\sigma^{\dag})-D^{\dag}\nabla \zeta(\sigma^{\dag}),\nabla P_H\varphi \right) =: \mathrm{I}_{2,1} +  \mathrm{I}_{2,2}. $$
	It is easy to observe
	\begin{align*}
		|  \mathrm{I}_{2,2} |           \le&  \| D_h^* - D^\dag \|_{L^2\II}  \|  \nabla \zeta(\sigma^\dag) \|_{L^\infty\II}  \| \nabla P_h\varphi\|_{L^2\II}
		\le c \alpha_2^{-\frac{1}{2}}\eta \epsilon.
	\end{align*}
	where we use the \textsl{a priori} estimate that $\|\nabla \zeta(\sigma^\dag) \|_{L^\infty\II} \le c$ and the assumption that $\|   D_h^* - D^\dag \|_{L^2\II} \le \epsilon$.
	Moreover, by Lemma \ref{lem:zhsh*-zs}, we have
	\begin{align*}
		|  \mathrm{I}_{2,1} |
		\le&  \|  D_h^*\|_{L^{\infty}\II} \|  \nabla \zeta_H(\sigma_H^*)- \nabla \zeta(\sigma^{\dag})\|_{L^2\II} \| \nabla P_h\varphi\|_{L^2\II}
		\le c \alpha_2^{-\frac{1}{2}}\eta \|  \nabla \zeta_H(\sigma_H^*)- \nabla \zeta(\sigma^{\dag})\|_{L^2\II}.
	\end{align*}
	By the inverse inequality \eqref{eq:inverse_ineq} and the approximation property \eqref{eq:error_P_h}, we have
	\begin{equation*}
		\begin{aligned}
			\|  \nabla \zeta_H(\sigma_H^*)- \nabla \zeta(\sigma^{\dag})\|_{L^2\II}
			\le&    \|  \nabla (\zeta_H(\sigma_H^*)-   P_H \zeta(\sigma^{\dag}))\|_{L^2\II} + \| \nabla (P_H \zeta(\sigma^{\dag})-  \zeta(\sigma^{\dag})) \|_{L^2\II} \\
			\le &   c \Big( H^{-1} \|  \zeta_H(\sigma_H^*)-   P_H \zeta(\sigma^{\dag}) \|_{L^2\II}
			+ H \|   \zeta(\sigma^{\dag})  \|_{H^2\II}\Big) \\
			\le & c   \big(  H^{-1} \eta + H \big).
		\end{aligned}
	\end{equation*}
	Meanwhile, using the stability estimate $\|  \nabla \zeta_H(\sigma_H^*)  \|_{L^2\II} + \| \nabla \zeta(\sigma^{\dag})  \|_{L^2\II} \le c$,
	we conclude that
	\begin{align*}
		|  \mathrm{I}_{2,1} |
		\le&   c \alpha_2^{-\frac{1}{2}}\eta  \min\big(  H^{-1} \eta + H, 1 \big).
	\end{align*}
	Thus we arrive at the estimate
	\begin{align}\label{eqn:est-sig-3}
		|  \mathrm{I}_{2}  |
		\le&   c \alpha_2^{-\frac{1}{2}}\eta  \big(\min\big(  H^{-1} \eta + H, 1 \big) + \epsilon\big).
	\end{align}
	Then the desired estimate follows immediately by combining the estimates \eqref{eqn:est-sig-1},  \eqref{eqn:est-sig-2} and  \eqref{eqn:est-sig-3}.
	Finally, if $f_2-f_1\ge c>0$, with the same argument as that of Theorem \ref{thm:stability_sigma}, we have the second assertion.
\end{proof}

\begin{remark}\label{rem:error_potential_sigma}
	According to Theorem \ref{thm:error_weighted_diffusion} and Remark \ref{rem:error_diffusion_D}, we have
	$$ \| D_h^*  - D^\dag   \|  \le \epsilon =   c\delta^{\frac{1}{4(1+\beta)}},$$
	provided that  $h \sim \sqrt\delta$, $\alpha_1 \sim \delta^2$ and the positivity condition \eqref{eq:positive_condition} is valid.
	As a result, with the choice of parameters $H\sim  \epsilon^{\frac{1}{2}}$ and $\alpha_2\sim   \epsilon^2$, there holds the estimate
	\begin{equation*}
		\| (\sigma^{\dag}-\sigma_H^*)\zeta (\sigma^{\dag})\|_{L^2\II}\le c \delta^{\frac{1}{8(1+\beta)}}.
	\end{equation*}
	Moreover, if we choose source terms such that $f_2 - f_1 \ge c >0$ in $\Omega$, then similar argument as that of Theorem \ref{thm:stability_sigma} implies that
	\begin{equation*}
		\| \sigma^{\dag}-\sigma_H^*\|_{L^2(\Omega')}\le c  \delta^{\frac{1}{8(1+\beta)}}
	\end{equation*}
	where $\Omega' \Subset \Omega$ and  the constant $c$ depending on $\mathrm{dist}(\Omega',\partial\Omega)$ and  $D^{\dag},\sigma^{\dag}$.
\end{remark}

\begin{remark}\label{rem:coupled_reconstruction}
Instead of adopting the aforementioned decoupled approach, alternative reconstruction formulas exist to address the inverse problem. One intuitive method involves the following least-squares formulation:
\begin{equation}\label{eqn:coupled_reconstruction}
\underset{D_h\in \mathcal{A}_{D,h},\sigma_h\in \mathcal{A}_{\sigma,h}  }{\mathrm{min}} \mathcal{J}(D_h,\sigma_h)=\frac{1}{2}\sum_{i=1}^{2}\|u_{i,h}-z_i^{\delta}\|_{L^2(\Omega)}^2+\frac{\alpha_1}{2}\|\nabla D_h\|_{L^2(\Omega)}^2+\frac{\alpha_2}{2}\|\nabla \sigma_h\|_{L^2(\Omega)}^2.
\end{equation}
where $u_{i,h}=u_{i,h}(D_h,\sigma_h)\in V_h$ satisfies $u_{i,h}(D_h,\sigma_h)|_{\partial\Omega} = \mathcal{I}_h g$ and
\begin{equation}\label{eqn:coupled_eq}
 ( D_h \nabla u_{i,h},\nabla v_h ) + (\sigma_h u_{i,h},v_h) = (f_i, v_h)\quad \forall~ v_h \in V_h^0.
\end{equation}
Due to the non-homogeneous boundary condition, deriving an error estimate similar to Theorem \ref{thm:error_weighted_diffusion} and \ref{thm:error_weighted_potential} for the coupled reconstruction formula \eqref{eqn:coupled_reconstruction}-\eqref{eqn:coupled_eq} and its numerical discretization is not feasible. Furthermore, a numerical comparison between the decoupled approach and the coupled approach is presented in Section \ref{sec:numer}.
\end{remark}

\section{Inverse problem for parabolic equation}\label{sec:parab}
In this section, we extend the argument to the following parabolic problem with exact conductivity $D^\dag$ and potential $\sigma^\dag$
\begin{equation}\label{eq:origin_parabolic}
	\left\{
	\begin{aligned}
		\partial_t u-\nabla \cdot (D^\dag \nabla u)+\sigma^\dag u&=f,\quad \text{in }\Omega\times (0,T],\\
		u&=g,\quad \text{on }\partial\Omega\times (0,T],\\
		u(0)&=u_0,\quad \text{in }\Omega.
	\end{aligned}
	\right.
\end{equation}
We aim to reconstruct the conductivity $D^\dag$ and the potential $\sigma^\dag$ by observing $u(x,t)$ for $(x,t) \in (T_i - \theta, T_i] \times \Omega$, where $i = 1, 2$. Here, $\theta > 0$ represents a small positive constant.

The inverse problem under consideration generally does not allow for unique recovery. To illustrate this, let's examine a one-dimensional example within the unit interval $\Omega = (0,1)$, where $g=0$ and $f=0$. Suppose $u_0(x) = \sin(\pi x)$. In this case, both of the following parameter sets yield identical solutions $u(x,t)$ for all $(x,t) \in \Omega \times [0,\infty)$:
\begin{itemize}
	\item[(1)] $D(x) = 1$ and $\sigma(x) = \pi^2$,
	\item[(2)] $D(x) = 2$ and $\sigma(x) = 0$;
\end{itemize}
As a result, the recovery process is highly sensitive to the choice of problem data. Therefore, it is necessary to impose certain assumptions on the problem data in the parabolic problem given by equation \eqref{eq:origin_parabolic}.

\begin{assumption}\label{assume:coeff_positive_para}
	The problem data in \eqref{eq:origin_parabolic} satisfy following properties.
	\begin{itemize}
		\item[(i)] The initial data $u_0\in W^{2,2}(\Omega)\cap W^{1,\infty}\II$ and $u_0 \ge c_0>  0$.
		\item[(ii)] The source data $f\in C^k(0,T;L^{\infty}(\Omega) )\cap C^{k+1}(0,T;L^{2}(\Omega) ) $ with $k\in \mathbb{N}^*$ and $f \ge 0$. Moreover, there exists $T_0>0$ such that
		\begin{equation}\label{eqn:f}
			f(x,t)=\left\{\begin{aligned}
				f_1(x)&,\quad 0<t\le T_0\\
				f_2(x)&,\quad t\ge2T_0,
			\end{aligned}\right.
		\end{equation}
		\item[(iii)] The boundary data $g\in  W^{\frac32,2}(\partial\Omega) \cap W^{1,\infty}(\partial\Omega)$ and $g\ge c_g>0$.
		\item[(iv)] The exact diffusion coefficient $D^\dag \in \A_D \cap W^{1,\infty}\II$ and the exact potential coefficient $\sigma^\dag \in \A_\sigma $.
	\end{itemize}
\end{assumption}
Under Assumption \ref{assume:coeff_positive_para}, the parabolic equation \eqref{eq:origin_parabolic} admits a unique solution $u \in L^{\infty}(0,T;  W^{1,\infty}(\Omega))$. Moreover, the parabolic maximum principle \cite[Section 7.1.4]{evans:book2022partial} implies that there exists a constant $\underline{c}_u$ depending on $D^\dag$, $\sigma^\dag$, $g$ and $\Omega$, but independent of $f$ such that
\begin{equation}\label{eqn:pos-parab}
	u(x,t)\ge \underline{c}_u  > 0, \quad \forall ~x,t\in \overline \Omega\times[0,T].
\end{equation}

\subsection{Stability estimate}
We first give \textsl{a priori} estimate which will be frequently used in the stability analysis.
We denote $A=A(D^\dag,\sigma^\dag)$ be the realization of $-\nabla\cdot D^\dag\nabla +\sigma^\dag$ with a zero Dirichlet boundary condition.
Note that the operator $A$ satisfies the following  resolvent estimate
\begin{equation}\label{eqn:resolv}
	\|  (z+A)^{-1} \|_{L^p\II\rightarrow L^p\II} \le c (1+|z|)^{-1},  \quad \forall z\in \Sigma_{\phi},
\end{equation}
where $\Sigma_{\phi}=\{0\ne z\in \mathbb{C}:|\arg(z)|\le \phi\}$ with a fixed $\phi\in (\pi/2,\pi)$.
If   $D^\dag \in \A_D$
and $\sigma^\dag \in \A_\sigma$, then the resolvent estimate \eqref{eqn:resolv} with $p=2$ holds for a constant
$c$ independent of $D^\dag$ and $\sigma^\dag$  (but depending on $\underline{c}_D$, $\bar{c}_D$ and
$\bar{c}_\sigma$ in \eqref{eq:admissible_set}). This could be easily proved by using a standard energy argument; see e.g.,
\cite[p. 92]{Thomee:2006}. Moreover, if $D^\dag \in W^{1,\infty}\II$, then  the resolvent estimate \eqref{eqn:resolv} holds for $p=\infty$ (cf. \cite[Theorem 1]{Stewart:1974}, \cite[Theorem 2.1]{Bakaev:2003} and \cite[Appendix  A]{LiMa:2022}).

Then we  introduce the solution operator
\begin{equation}\label{op:E}
	E(t)=\frac{1}{2\pi i}\int_{\Gamma_{\phi,\kappa }}e^{zt}(z+A)^{-1}\d z.
\end{equation}
Here $\Gamma_{\phi,\kappa}=\{z\in \mathbb{C}:|z|=\kappa,|\arg(z)|\le \phi\}\cup\{z\in \mathbb{C}:z=\rho e^{i\phi},\rho\ge \kappa\}$
with fixed constants $\kappa\in (0,\infty)$ and $\phi\in(\pi/2,\pi)$.
Then the solution $u$ to the parabolic problem \eqref{eq:origin_parabolic} could be written as
\begin{equation}\label{eqn:sol_reps}
	u(t)=E(t)(u_0-\bar{g})+\int_0^t E(s)f(t-s)\d s,
\end{equation}
where $\bar{g}$ satisfies elliptic equation $A\bar{g}=0$ with Dirichlet boundary condition $\bar g|_{\partial\Omega} = g$.
Under Assumption \ref{assume:coeff_positive_para}, by the elliptic regularity theory and maximum principle, $\bar{g}\in W^{1,\infty}(\Omega)\cap W^{2,2}(\Omega)$ and has a strictly positive lower  bound.

Next, we provide a selection of valuable smoothing properties associated with the solution operator $E(t)$.

\begin{lemma}\label{lem:sm_prop_E}
	Let $E(t)$ be the solution opetator defined in \eqref{op:E}.  Suppose that  $D^\dag \in \A_D$
	and $\sigma^\dag \in \A_\sigma$. Then there exists a constant $c$ independent of $D^\dag$ and $\sigma^\dag$
	(but depending on $\underline{c}_D$, $\bar{c}_D$ and
	$\bar{c}_\sigma$ in \eqref{eq:admissible_set}) such that for any nonnegative integer $\ell$,
	\begin{equation*}
		\|   E^{(\ell)}(t)\|_{L^{2}\II\rightarrow L^2\II}\le c \min(t^{-\ell-1},t^{-\ell}).
	\end{equation*}
	Moreover, if $D^\dag \in W^{1,\infty}\II$, then there holds
	\begin{equation*}
		\|   E^{(\ell)}(t)\|_{L^{\infty}\II\rightarrow L^{\infty}\II}\le c \min(t^{-\ell-1},t^{-\ell}).
	\end{equation*}
\end{lemma}
\begin{proof}
The proof of the lemma follows by the contour integral \eqref{op:E} and the resolvent estimate \eqref{eqn:resolv}. If  $D^\dag \in \A_D$ and $\sigma^\dag \in \A_\sigma$, then the resolvent estimate \eqref{eqn:resolv} with $p=2$ holds for a constant $c$ depending on $\underline{c}_D$, $\bar{c}_D$ and
$\bar{c}_\sigma$ in \eqref{eq:admissible_set}. Then for $\ell \ge 0$ we have
\begin{align*}
		\|  E^{(\ell)}(t)\|_{L^2\II\rightarrow  L^2\II}&\le c\int_{\Gamma_{\phi,\kappa }}|z|^\ell |e^{z t}|\| (z+A)^{-1}\|_{L^{2}\II\rightarrow  L^{2}\II} |\d z|
		\le C\int_{\Gamma_{\phi,\kappa }}|e^{z t}| |z|^\ell (1+|z|)^{-1}  |\d z| \\
		&  \le c \Big(\int_ \kappa^\infty e^{- st \cos \phi}\frac{s^\ell}{s+1}  \,\d s +  \int_{-\phi}^\phi e^{- \kappa t \cos\psi} \frac{\kappa^{\ell+1}}{1+\kappa}  \,\d \psi \Big)
		\le c \min(t^{-\ell-1},t^{-\ell}).
\end{align*}
where we take $\kappa=t^{-1}$ in the last inequality.
The estimate in maximum-norm could be derived similarly using the resolvent estimate \eqref{eqn:resolv} with $p=\infty$.
\end{proof}\vskip5pt
Lemma  \ref{lem:sm_prop_E} immediately leads to the next lemma showing the decay properties of the solution. \vskip5pt

\begin{lemma}\label{lem:dtu}
	Suppose that Assumption \ref{assume:coeff_positive_para} (i)-(iii) holds valid.
	Let $u$ be the solution of the parabolic problem \eqref{eq:origin_parabolic}.
	If $D^\dag \in \A_D$
	and $\sigma^\dag \in \A_\sigma$, then  for $\ell=1,\cdots,k+1$,  there holds
	$$\|   \partial_t^{\ell}u(t)\|_{L^{2}\II} \le c \max(t^{-\ell},1),\quad \forall ~t\in  (0,\infty).$$
	Moreover, if $D^\dag \in W^{1,\infty}\II$, then there holds
	\begin{equation*}
		\|  \partial_t u(t)\|_{L^{\infty}\II} \le
		\begin{cases}
			c t^{-1},  &\quad \forall~ t\in (0,T_0],  \\
			c(t-2T_0)^{-1},& \quad \forall ~t\in  (2T_0,\infty).
		\end{cases}
	\end{equation*}
\end{lemma}
\begin{proof}
	By solution representation \eqref{eqn:sol_reps}, we may write
	\begin{equation*}
		\partial_{t}^{\ell} u(t)=E^{(\ell)}(t)(u_0-\bar{g})+\sum_{i=0}^{\ell-1}E^{(i)}(t)f^{(\ell-1-i)}(0) +\int_0^t E(s) f^{(\ell)}(t-s)\d s.
	\end{equation*}
	According to Assumption \ref{assume:coeff_positive_para} (iii), we have $ f^{(i)}(0)=0$, $i=1,\cdots,\ell-1$. Then Lemma \ref{lem:sm_prop_E}  implies
	\begin{align*}
		\| \partial_{t}^{\ell} u(t)\|_{L^{2}\II} \le&  \|  E^{(\ell)}(t)\|_{L^{2}\II\rightarrow L^{2}\II} \|   u_0-\bar{g}  \|_{L^{2}\II}+ \|  E^{(\ell-1)}(t)\|_{L^{2}\II\rightarrow L^{2}\II} \|   f(0) \|_{L^{2}\II}\\
		&+\int_0^t\|E(s)\|_{L^{2}\II\rightarrow L^{2}\II}\|  f^{(\ell)}(t-s)\|_{L^{2}\II}\d s\\
		\le& ct^{-\ell}+ct^{-\ell+1}+c\le c\max(t^{-\ell},1).
	\end{align*}
	
	Now we assume that $D^\dag \in W^{1,\infty}\II$. Similarly, with solution representation \eqref{eqn:sol_reps}, $\partial_t u$ could be written as
	\begin{equation*}
		\partial_t u(t)=E'(t)(u_0-\bar{g})+E(t)f(0)+\int_0^t E(s)\partial_t f(t-s)\d s
	\end{equation*}
	For any $t \in (0,T_0]$, we recall that $\partial_t f=0$ according to Assumption \ref{assume:coeff_positive_para}.
	Then the regularity of problem data in Assumption \ref{assume:coeff_positive_para} and
	Lemma \ref{lem:sm_prop_E}  lead to
	\begin{align*}
		\|  \partial_t u(t)\|_{L^{\infty}\II} \le \|E'(t)\|_{L^{\infty}\II\to  L^{\infty}\II}\|u_0-\bar{g}\|_{L^{\infty}\II}+\|E(t)\|_{L^{\infty}\II\to  L^{\infty}\II}\|f(0)\|_{L^{\infty}\II}\le c t^{-1}.
	\end{align*}
	Next, we turn to the case that $t \in (2T_0,\infty)$.
	Assumption \ref{assume:coeff_positive_para} leads to $\partial_t f\ne 0$ for $t\in (T_0,2T_0)$ and $\partial_t f \equiv 0$ otherwise. Then Lemma \ref{lem:sm_prop_E}   yields
	\begin{align*}
		\|  \partial_t u(t)\|_{L^{\infty}\II} \le&  \|  E'(t)\|_{L^{\infty}\II\rightarrow L^{\infty}\II} \|   u_0-\bar{g}  \|_{L^{\infty}\II}+ \|  E(t)\|_{L^{\infty}\II\rightarrow L^{\infty}\II} \|   f(0) \|_{L^{\infty}\II}\\
		&+\int_0^t\|E(s)\|_{L^{\infty}\II\rightarrow L^{\infty}\II}\|\partial_t f(t-s)\|_{L^{\infty}\II}\d s\\
		\le& Ct^{-1}+C\int_{t-2T_0}^{t-T_0} s^{-1}\d s\le C(t-2T_0)^{-1}.
	\end{align*}
	This completes the proof of the lemma.
\end{proof}

The stability estimation follows a similar approach to that of the elliptic case. First, we decouple the parameters, and then sequentially determine the stability for the diffusion coefficient $D$ and the potential $\sigma$. To achieve this, we select time intervals such that $0 < T_1 \le T_0$ and $T_2 \ge 2T_0$.  Multiplying the equation at time $T_1$ by $u(T_2)$ and at time $T_2$ by $u(T_1)$, after subtracting the two equations, we can eliminate the potential $\sigma$ and obtain
\begin{equation}\label{eq:parabolic_diffusion}
	\left\{
	\begin{aligned}
		-\nabla \cdot \left(Du^2(T_1)\nabla\left( \frac{u(T_2)}{u(T_1)}-1\right)\right)&=(f(T_2)-\partial_tu(T_2))u(T_1)-(f(T_1)-\partial_tu(T_1))u(T_2),\quad \text{in }\Omega\\
		\frac{u(T_2)}{u(T_1)}-1&=0,\quad \text{on }\partial\Omega.
	\end{aligned}
	\right.
\end{equation}
For ease of reference, let's introduce the following notation:
\begin{equation}\label{eqn:parab-notation}
	w:=\frac{u(T_2)}{u(T_1)}-1,\quad q:=D|u(T_1)|^2\quad \text{and}\quad F:=(f(T_2)-\partial_tu(T_2))u(T_1)-(f(T_1)-\partial_tu(T_1))u(T_2).
\end{equation}
Then the system \eqref{eq:parabolic_diffusion} could be written as the form \eqref{eq:elliptic_diffusion_0_bdry}.
Therefore, the following result is an immediate application of Theorem \ref{thm:stab-Du2} and we omit the proof.
\begin{assumption}\label{ass:para-2}
	The exact diffusion coefficient $q^{\dag} = D^\dag|u(T_1)|^2 \in W^{2,2}\II\cap W^{1,\infty}\II\cap \mathcal{A}_q$ and source term $F=(f(T_2)-\partial_tu(T_2))u(T_1)-(f(T_1)-\partial_tu(T_1))u(T_2)  \in L^{\infty}\II$.
\end{assumption}
\begin{theorem}\label{thm:stability_para_Du_1^2}
Suppose that  $F$, $\tilde{F}\in L^{\infty}\II$, $q$ satisfy Assumption \ref{ass:para-2} and $\tilde q\in \mathcal{A}_q$.
Also, suppose the $W^{1,2}\II$-norm of $q$ and $\tilde q$ are bounded by a generic constant $C$.
Let $w$ be the solution of \eqref{eq:parabolic_diffusion} with diffusion coefficient $q$ and source $F$, while  $\tilde{w}$ be the solution  with diffusion coefficient $\tilde{q}$ and source $\tilde{F}$. Then there holds
	\begin{align*}
		\int_{\Omega}\frac{(q-\tilde q)^2}{q^2}\left(q\left|\nabla w\right|^2+F w\right)\d x\le c\left(\|  w-\tilde w\|_{W^{1,2}(\Omega)}+\|  F-\tilde F\|_{L^2(\Omega)}\right).
	\end{align*}
	Moreover, if  the following positive condition holds
	\begin{align}\label{eq:positive_condition_para_D}
		q\left|\nabla w \right|^2+Fw\ge c\,\mathrm{dist}(x,\partial \Omega)^{\beta} \quad a.e. \text{ on }\Omega
	\end{align}
	for some $\beta\ge 0$ and $c>0$. Then the following estimate holds
	\begin{align*}
		\|  q-\tilde q\|_{L^2(\Omega)}\le c\left(\|  w-\tilde w\|_{W^{1,2}(\Omega)}+\|  F-\tilde F\|_{L^2(\Omega)}\right)^{\frac{1}{2(1+\beta)}}
	\end{align*}
\end{theorem}

\begin{remark}\label{rmk:positive_condtion_hold_para}
As suggested by Remark \ref{rmk:positive_condtion_hold}, the crucial aspect to ensure the positivity condition in \eqref{eq:positive_condition_para_D} is to establish the strict positivity of the function $F$. It is worth noting that $F$ contains the time derivative term $\partial_tu$, making the positive condition more intricate than that presented in Remark \ref{rmk:positive_condtion_hold}. Therefore, a careful examination is required.
Next, we demonstrate that the positive condition prevails under certain specific excitations $f$ and $g$.

We claim that the function $F$ in \eqref{eqn:parab-notation} could be strictly positive provided some restrictions on excitation $f$ and $g$. For example, with Assumption \ref{assume:coeff_positive_para}, we take a time dependent source term $f(x,t)$ which has a form of
\begin{equation*}
		f(x,t)=\left\{
		\begin{aligned}
			0&,\quad 0<t\le T_0\\
			c_f&,\quad t>2T_0,
		\end{aligned}
		\right.
\end{equation*}
for some $T_0$. Here we assume that $T_0$ is sufficient large, $T_1=T_0$, and $T_2 = 3T_0$. As a result, according to Lemma \ref{lem:dtu}, $\|\partial_t u(T_1)\|_{L^{\infty}\II} + \|\partial_t u(T_2)\|_{L^{\infty}\II}\le CT_0^{-1}$. Since $u$ has a strict positive lower bound  \eqref{eqn:pos-parab} and $\| u (t)\|_{L^\infty\II}\le c$ uniform in $t$, we conclude that for sufficiently large $T_0$, $$F \ge c_f \underline{c}_u-T_0^{-1}(\underline{c}_u + \|  u(T_2)  \|_{L^\infty\II})> c_F > 0.$$   Then the positivity condition \eqref{eq:positive_condition_para_D} holds valid for some $\beta\in[0,2]$.
\end{remark}\vskip5pt


The following corollary gives the estimate of $\|  F-\tilde F\|_{L^2\II}$ and hence the estimate of $\|  D-\tilde D\|_{L^2\II}$.
\begin{corollary}\label{coro:para_F-tF}
	Suppose that Assumption \ref{assume:coeff_positive_para} holds valid, $\tilde D \in \mathcal{A}_D$ and $\sigma\in \mathcal{A}_\sigma$.
	Let $u$ ($\tilde u$) be the solution to the parabolic equation \eqref{eqn:PDE-parab} with diffusion coefficient $D$ ($\tilde D$),
	potential coefficient $\sigma$ ($\tilde \sigma$), the boundary data $g$ and source $f$.
	Let $D\in W^{1,\infty}\II\cap \mathcal{A}_D$, $\sigma\in L^{\infty}\II\cap \mathcal{A}_{\sigma}$. Assume that $\| \tilde u(T_i) \|_{L^2\II} \le C$ for $i=1,2$.
	Then there holds
	\begin{align*}
		\|  F-\tilde F\|_{L^2\II}\le c \Big(\sum_{i=1}^2 \|  u-\tilde u\|_{L^{\infty}(T_i-\theta,T_i;L^2\II )}\Big)^{\frac{k}{k+1}}.
	\end{align*}
	If in addition, the positive condition \eqref{eq:positive_condition_para_D} holds for some $\beta\ge 0$, then
	\begin{align*}
		\|  D-\tilde D\|_{L^2\II}\le c\left(\sum_{i=1}^2\|(u-\tilde u)(T_i)\|_{W^{1,2}\II}+
		\Big(\sum_{i=1}^2 \|  u-\tilde u\|_{L^{\infty}(T_i-\theta,T_i;L^2\II )}\Big)^{\frac{k}{k+1}}\right)^{\frac{1}{2(1+\beta)}}.
	\end{align*}
\end{corollary}
\begin{proof}
	By definition in \eqref{eqn:parab-notation}, $F-\tilde F$ can be written as
	\begin{align*}
		F-\tilde F&=f(T_2) \left( u(T_1)-\tilde u(T_1) \right)  +f(T_1) \left( \tilde u(T_2)-u(T_2) \right)\\
		&+u(T_2)\partial_tu(T_1)-\tilde u(T_2)\partial_t\tilde u(T_1)+\tilde u(T_1)\partial_t \tilde u(T_2)- u(T_1)\partial_t u(T_2)
	\end{align*}
	The first two terms can be easily bounded by
	$$\| f(T_2) \left( u(T_1)-\tilde u(T_1) \right)  +f(T_1) \left( \tilde u(T_2)-u(T_2) \right) \|_{L^2\II} \le  c \sum_{i=1}^2\|(u-\tilde u)(T_i)\|_{L^2\II}.$$
	Inserting an intermediate term $\tilde u(T_2)\partial_t u(T_1)$, we have
	\begin{align*}
		\|  u(T_2)\partial_tu(T_1)-\tilde u(T_2)\partial_t\tilde u(T_1)\|_{L^2\II}&\le \|  \tilde u(T_2)\|_{L^{\infty}\II}\|  \partial_tu(T_1)-\partial_t\tilde u(T_1)\|_{L^2\II}\\
		&+ \| \partial_t  u(T_1) \|_{L^{\infty}\II}\|  u(T_2)- \tilde u(T_2) \|_{L^2\II }
	\end{align*}
	By assumption, we have $\|\tilde u(T_2)\|_{L^{\infty}\II}\le C$. Meanwhile, Lemma \ref{lem:dtu}  implies $\|  \partial_t  u(T_1)\|_{L^{\infty}\II}\le CT_1^{-1}\le C$.
	These together imply
	\begin{align*}
		\|  u(T_2)\partial_tu(T_1)-\tilde u(T_2)\partial_t\tilde u(T_1)\|_{L^2\II}&\le C\left(\| \partial_tu(T_1)-\partial_t\tilde u(T_1)\|_{L^2\II}	+ \|  u(T_2)- \tilde u(T_2) \|_{L^2\II }\right).
	\end{align*}
	To analyze the term $ \|  \partial_tu(T_1)-\partial_t\tilde u(T_1)\|_{L^2\II}$, we insert the backward difference quotient of order $k$,
	\begin{equation*}
		\partial_\tau u(T_1) = \tau^{-1} \sum_{j=0}^{k}a_{j}u(T_1-j\tau)\quad \text{and}\quad \partial_\tau \tilde{u}(T_1) = \tau^{-1} \sum_{j=0}^{k}a_{j}\tilde{u}(T_1-j\tau),
	\end{equation*}
	for some $0<\tau <\theta/k$, where $\{a_j\}_{j=1}^{k}$  are backward difference quotient coefficients. By Lemma \ref{lem:dtu}, we have
	$ \|  \partial_{t}^{k+1} u  \|_{L^\infty(T_1-\theta,T_1;L^2\II)}\le c $, and hence by Taylor's expansion we obtain
	\begin{equation}\label{eqn:err-quot}
		\begin{aligned}
			\|  \partial_t u(T_1)-\partial_\tau u(T_1)\|_{L^2\II}\le c\tau^k  \|  \partial_{t}^{k+1} u  \|_{L^\infty(T_1-\theta,T_1;L^2\II)} \le c \tau^k.
		\end{aligned}
	\end{equation}
	Then we obtain
	\begin{align*}
		\| u(T_2)\partial_tu(T_1)-\tilde u(T_2)\partial_t\tilde u(T_1)\|_{L^2\II}&\le c\left(\tau^k+\tau^{-1}\| u-\tilde u\|_{L^{\infty}(T_1-\theta,T_1;L^2\II )}
		+\| (u-\tilde u)(T_2)\|_{L^2\II}\right).
	\end{align*}
	The bound for $\| u(T_1)\partial_tu(T_2)-\tilde u(T_1)\partial_t\tilde u(T_2)\|_{L^2\II}$ can be obtained similarly.
	Consequently, we arrive  at
	\begin{align*}
		\|  F-\tilde F\|_{L^2\II}\le c\Big(\tau^k + \tau^{-1}\sum_{i=1}^2 \|  u-\tilde u\|_{L^{\infty}(T_i-\theta,T_i;L^2\II )}\Big)
	\end{align*}
	Then the choice $\tau = \left(\sum_{i=1}^2 \|  u-\tilde u\|_{L^{\infty}(T_i-\theta,T_i;L^2\II )}\right)^{\frac{1}{k+1}}$ leads to the desired estimate.
	The second assertion of the corollary is a direct consequence of applying Theorem \ref{thm:stability_para_Du_1^2}.
\end{proof}

Having obtained the reconstructed conductivity, we can now proceed to recover the potential. The subsequent theorem offers conditional stability for the potential's recovery.
\begin{theorem}\label{thm:stability_sigma_para}
	Let the assumptions in Corollary \ref{coro:para_F-tF} hold true. Then there holds
	\begin{align*}
		\|  (\sigma-\tilde\sigma) & (u(T_2)-u(T_1))\|_{L^2\II}\\
		&\le c\left( \Big(\sum_{i=1}^2 \|  u-\tilde u\|_{L^{\infty}(T_i-\theta,T_i;L^2\II )}\Big)^{{\frac{2k}{k+1}}} + \sum_{i=1}^2 \|  (u-\tilde u)(T_i)\|_{W^{1,2}\II} +\|  D-\tilde D\|_{L^2\II} \right)^{\frac{1}{2}}.
	\end{align*}
\end{theorem}
\begin{proof}
Letting $\zeta=u(T_2)-u(T_1)$ and choosing a test function $v\in H_0^1\II$, we have
\begin{align*}
	((\sigma-\tilde\sigma)\zeta,v)=(\tilde D\nabla \tilde \zeta-D\nabla \zeta,\nabla v)+(\tilde\sigma(\tilde \zeta-\zeta),v)+\sum_{i=1}^2(\partial_t(\tilde u- u)(T_1),v).
\end{align*}
Now we take $v=(\sigma-\tilde\sigma)\zeta$ and note that $\| \nabla v \|_{L^2\II} \le c$. Then we obtain
\begin{equation*}
	\begin{split}
	\|  v\|_{L^2\II}^2 &\le c\Big(\|   \zeta-\tilde \zeta\|_{W^{1,2}\II}+\|  D-\tilde D\|_{L^2\II}+\|  \tilde \zeta-\zeta\|_{L^2\II}\|  v\|_{L^2\II}
	+\sum_{i=1}^2\| \partial_t(\tilde u- u)(T_i)\|_{L^2\II} \|  v\|_{L^2\II} \Big).
	\end{split}
\end{equation*}
	Using the argument in the proof of Corollary \ref{coro:para_F-tF}, we use backward difference quotient to estimate $\| \partial_t(\tilde u- u)(T_i)\|_{L^2\II} $ and obtain
	\begin{align*}
		\sum_{i=1}^2\| \partial_t(\tilde u- u)(T_i)\|_{L^2\II}   \le \Big( \sum_{i=1}^2 \|  u-\tilde u\|_{L^{\infty}(T_i-\theta,T_i;L^2\II )}\Big)^{\frac{k}{k+1}}.
	\end{align*}
As a result, we conclude that
\begin{equation*}
\begin{split}
	\|  v\|_{L^2\II}^2 &\le c\Big(\sum_{i=1}^2\| (u - \tilde u)(T_i)\|_{W^{1,2}\II}+\|D-\tilde D\|_{L^2\II} +\Big( \sum_{i=1}^2 \|  u-\tilde u\|_{L^{\infty}(T_i-\theta,T_i;L^2\II )}\Big)^{\frac{2k}{k+1}}\Big).
\end{split}
\end{equation*}
	This completes the proof of the theorem.
\end{proof}

\begin{remark}\label{rem:sigma_internal_error}
	Similar as in Theorem \ref{thm:stability_sigma}, for any compact subset $\Omega'\Subset \Omega$ with $\mathrm{dist}(\overline{\Omega'},\partial\Omega)>0$, we can obtain the bound for $\|\sigma-\tilde{\sigma}\|_{L^2(\Omega')}$, if $\zeta=u(T_2)-u(T_1)\ge C>0$ in $\Omega'$.
	This condition can be achieved by following choice of data: we take $f_2-f_1\ge c_f>0$ in \eqref{eqn:f}, $T_1=T_0$, $T_2=3T_0$, with $T_0$ sufficiently large.
	Then $\zeta$ satisfies the elliptic equation with source $F=f_2-f_1+\partial_t u(T_1)-\partial_t u(T_2) $. According to Lemma \ref{lem:dtu}, there holds
	$\|\partial_t u(T_1)\|_{L^{\infty}\II} + \|\partial_t u(T_2)\|_{L^{\infty}\II}\le c T_0^{-1}$, we conclude that $F\ge c_f-cT_0^{-1}$ is strictly positive when $T_0$ is sufficiently large.
	Consequently, the strong maximun principle \cite[Theorem 1]{Vazquez:1984} implies $\zeta \ge c>0$ in $\Omega'$ and hence
	\begin{align*}
		\|  \sigma-\tilde\sigma\|_{L^2(\Omega')}
		\le c\Big( \big(\sum_{i=1}^2 \|  u-\tilde u\|_{L^{\infty}(T_i-\theta,T_i;L^2\II )}\big)^{\frac{2k}{k+1}} + \sum_{i=1}^2 \|  (u-\tilde u)(T_i)\|_{W^{1,2}\II} +\|  D-\tilde D\|_{L^2\II} \Big)^{\frac{1}{2}}.
	\end{align*}
\end{remark}

\subsection{Numerical scheme and error analysis}
In this part, we present the numerical scheme for the   reconstruction  of diffusion coefficient and potential coefficient.
First of all, we use the system \eqref{eq:parabolic_diffusion} to recover the diffusion coefficient without the knowledge of potential.
Denote the exact solution $u^\dag(x,t):=u(x,t;D^\dag,\sigma^\dag)$ and define
\begin{equation*}
	w^{\dag}=\frac{u(T_1)}{u(T_2)}-1  \quad \text{and}\quad w^{\delta}=\frac{z^{\delta}(T_1)}{z^{\delta}(T_2)}-1
\end{equation*}
Recall that the exact solution $u^\dag$ is strictly positive, cf. \eqref{eqn:pos-parab}.
Then for ease of simplicity, we assume that $z^{\delta}$ are strictly positive in $\overline\Omega$.
Moreover, we assume that
\begin{equation}\label{eqn:bound-z}
	\| z^\delta \|_{C((T_i-\theta_i,T]; L^\infty(\Omega))} \le c
\end{equation}
with some generic constant $c$.
Then it is easy to observe
\begin{equation*}
	\|w^{\delta}-w^{\dag}\|\le c\delta.
\end{equation*}
Moreover, we take
$$F^{\delta}=\left(f(T_2)-\partial_{\tau}z^{\delta}(T_2)\right)z^{\delta}(T_1)-\left(f(T_1)-\partial_{\tau} z^{\delta}(T_1)\right)z^{\delta}(T_2),$$
where $ \partial_{\tau}$ denote backward difference quotient of orde $k$ for some $0<\tau<\theta/k$.
We apply the following output least squares formulation
\begin{align}\label{eq:min_D_dis_para}
	\min_{q_h\in \mathcal{A}_{q,h}} J_{\alpha_1,h }(q_h)=\frac{1}{2}\|  w_h(q_h)-w^{\delta}\|_{L^2\II}^2+\frac{\alpha_1}{2}\|  \nabla q_h\|_{L^2\II}^2
\end{align}
where $w_h(q_h)\in V_h^0$ is a weak solution of
\begin{equation}\label{eq:min_D_dis_restriction_para}
	(q_h\nabla w_h,\nabla v_h)=(F^{\delta},v_h),\quad\forall v_h\in V_{h}^0.
\end{equation}
The following theorem is a direct consequence of Theorem \ref{thm:error_weighted_diffusion}.
\begin{theorem}\label{thm:error_weighted_diffusion_para}
	Suppose Assumption \ref{assume:coeff_positive_para} and \ref{ass:para-2} hold. Let $q^{\dag} = D^\dag |u^\dag(T_1)|^2\in \mathcal{A}_q$ be the exact parameter in
	the elliptic equation \eqref{eq:parabolic_diffusion}, $w(q^{\dag})$ be the exact solution, and $q_h^*\in \mathcal{A}_{q,h}$ be a minimizer of problem \eqref{eq:min_D_dis_para}-\eqref{eq:min_D_dis_restriction_para}.Then with $\eta=h^2+\delta+\tau^k + \delta\tau^{-1}+\alpha_1^{\frac{1}{2}}$ , there holds
	\begin{align*}
		\int_{\Omega}\left(\frac{q^{\dag}-q_h^*}{q^{\dag}}\right)^2\left( q^{\dag}|\nabla w(q^{\dag})|^2+Fw(q^{\dag})\right)\d x\le c \left( \left(h\eta\alpha_1^{-\frac{1}{2}}+\min(h+h^{-1}\eta,1)\right)\eta\alpha_1^{-\frac{1}{2}}+\delta+\tau^k+\delta\tau^{-1}\right)
	\end{align*}
	Moreover, if  the following positive condition \eqref{eq:positive_condition_para_D} holds for $\beta\ge 0$, we have
	\begin{align*}
		\| q^{\dag}-q_h^*\|_{L^2(\Omega)}\le c \left( \left(h\eta\alpha_1^{-\frac{1}{2}}+\min(h+h^{-1}\eta,1)\right)\eta\alpha_1^{-\frac{1}{2}}+
		\delta+\tau^k+\delta\tau^{-1}\right)^{\frac{1}{2(1+\beta)}}
	\end{align*}
\end{theorem}
\begin{proof}
	The proof closely resembles that of Theorem \ref{thm:error_weighted_diffusion}, with the primary distinction being the estimation of $F - F^{\delta}$. More specifically, we obtain
	\begin{align*}
		\|  F- F^{\delta}\|_{L^2\II}&
		=\|  f(T_2)\big(u(T_1)-z^{\delta}(T_1)\big)\|_{L^2\II}+\|  f(T_1)\big(u(T_2)-z^{\delta}(T_2)\big)\|_{L^2\II}\\
		&\quad + \|  \partial_t u(T_1) u(T_2)-\partial_{\tau} z^{\delta}(T_1) z^{\delta}(T_2) \|_{L^2\II}+\|  \partial_t u(T_2) u(T_1)-\partial_{\tau} z^{\delta}(T_2) z^{\delta}(T_1)\|_{L^2\II}\\
		&\le  c \delta + \|  \partial_t u(T_1) u(T_2)-\partial_{\tau} z^{\delta}(T_1) z^{\delta}(T_2) \|_{L^2\II}+\|  \partial_t u(T_2) u(T_1)-\partial_{\tau} z^{\delta}(T_2) z^{\delta}(T_1)\|_{L^2\II}.
	\end{align*}
	Then it suffices to bound the second term, and then the third term follows analogously.
	Inserting the terms $\partial_\tau u(T_1) u(T_2)$ and $\partial_\tau z^\delta(T_1) u(T_2)$, we have
	\begin{equation*}
		\begin{aligned}
			&\quad  \|  \partial_t u(T_1) u(T_2)-\partial_{\tau} z^{\delta}(T_1) z^{\delta}(T_2) \|_{L^2\II} \\
			&\le \|  \partial_t u(T_1) u(T_2)-\partial_\tau u(T_1) u(T_2) \|_{L^2\II}
			+  \| \partial_\tau u(T_1) u(T_2) - \partial_{\tau} u(T_1) z^\delta(T_2) \|_{L^2\II} \\
			& + \| \partial_{\tau} u(T_1) z^\delta(T_2) - \partial_{\tau} z^{\delta}(T_1) z^{\delta}(T_2)\|_{L^2\II} = \sum_{j=1}^3 \mathrm{I}_j.
		\end{aligned}
	\end{equation*}
	By Lemma \ref{lem:dtu}, we observe that $\|\partial_t^{k+1}u \|_{L^\infty(T_i-\theta,T_i;L^2\II)}  + \|   u(T_i) \|_{L^\infty\II} \le c$
	and hence we apply the estimate \eqref{eqn:err-quot} to obtain
	\begin{equation*}
		\mathrm{I}_1 \le  \|  \partial_t u(T_1) -\partial_\tau u(T_1)   \|_{L^2\II}   \|   u(T_2) \|_{L^\infty\II}
		\le c \tau^k \|\partial_t^{k+1}u \|_{L^\infty(T_1-\theta,T_1;L^2\II)} \|   u(T_2) \|_{L^\infty\II}
		\le c \tau^k.
	\end{equation*}
	For the second term, we apply Lemma \ref{lem:dtu} and the assumption \eqref{eq:noise_level_para} to obtain
	\begin{equation*}
		\begin{aligned}
			\mathrm{I}_2 &\le  \| \partial_\tau u(T_1) \|_{L^\infty\II}  \|  u(T_2)  - z^\delta \|_{L^2\II}
			\le c  \delta  \tau^{-1}   \int_{T_1 - \theta}^{T_1} \|u_t(t)\|_{L^\infty\II} \,\d t
			\le c \delta \tau^{-1}.
		\end{aligned}
	\end{equation*}
	Finally, for the term $\mathrm{I}_3$, we use the assumption \eqref{eq:noise_level_para} and \eqref{eqn:bound-z} to derive
	\begin{equation*}
		\begin{aligned}
			\mathrm{I}_3
			&\le  \| \partial_{\tau} u(T_1) z^\delta(T_2) - \partial_{\tau} z^{\delta}(T_1) z^{\delta}(T_2)\|_{L^2\II} \\
			&\le c \|  \partial_{\tau} (u(T_1) - z^\delta(T_1)) \|_{L^2\II} \| z^{\delta}(T_2) \|_{L^\infty\II} \le c \delta \tau^{-1}.
		\end{aligned}
	\end{equation*}
	Consequently, we arrive at
	\begin{equation*}
		\|  F- F^{\delta}\|_{L^2\II} \le c (\delta + \tau^k+\delta \tau^{-1}).
	\end{equation*}
	The proof that follows simply involves substituting $\| F- F^{\delta}\|_{L^2\II}$ in Theorem \ref{thm:error_weighted_diffusion} with the newly established error bound. As such, we omit this largely redundant proof.
\end{proof}

\begin{remark}\label{rem:error_diffusion_D_para}
	In Theorem \ref{thm:error_weighted_diffusion_para}, the choice $\tau\sim \delta^{\frac{1}{k+1}}$ implies that $\|  F- F^{\delta}\|_{L^2\II} \le c \delta^{\frac{k}{k+1}}$.
	With \textsl{a priori} choice of the algorithmic parameters: $h \sim  \delta^{\frac{k}{2(k+1)}}$, 	$\alpha_1 \sim \delta^{\frac{2k}{k+1}}$. Under the positivity condition \eqref{eq:positive_condition_para_D} with $\beta\ge0$, there holds the estimate
	\begin{align*}
		\|  D^{\dag}-D_h^*\|_{L^2\II}\le c\, \delta^{\frac{k}{4(1+\beta)(k+1)}}.
	\end{align*}
	When $f$ is smooth over time, the optimal order is nearly $(4(1+\beta))^{-1}$. This estimation aligns with the one given in the elliptic case, as detailed in Remark \ref{rem:error_diffusion_D}.
	\end{remark}\vskip5pt

Now we turn to the reconstruction formula of potential $\sigma^\dag$ and study the approximation error. Let $\zeta=u(T_1)-u(T_2)$ which is a solution to the  elliptic problem
\begin{equation}\label{eq:para_potential_0_bdry}
	\left\{
	\begin{aligned}
		-\nabla \cdot \left(D \nabla \zeta\right)+\sigma \zeta&=f(T_1)-f(T_2)+\partial_t u(T_2)-\partial_t u(T_1),\quad \text{in }\Omega,\\
		\zeta&=0,\quad \text{on }\partial\Omega.
	\end{aligned}
	\right.
\end{equation}
The noisy observational data for equation \eqref{eq:para_potential_0_bdry} is
$\zeta^{\delta}=z^{\delta}(T_2)-z^{\delta}(T_1)$ satisfying $\|\zeta-\zeta^{\delta}\|_{L^2(\Omega)}\le c\delta$. We consider following least-squares formulation
\begin{align}\label{eq:min_sigma_dis_para}
	\min_{\sigma_H\in \mathcal{A}_{\sigma,H}} \mathcal{J}_{\alpha_2,H }(\sigma_H)=\frac{1}{2}\|  \zeta_H(\sigma_H)-\zeta^{\delta}\|_{L^2\II}^2+\frac{\alpha_2}{2}\|  \nabla \sigma_H\|_{L^2\II}^2
\end{align}
where $\mathcal{A}_{\sigma,H}=\mathcal{A}_{\sigma}\cap V_H$ and $\zeta_H(\sigma_H)\in V_H^0$ is the solution to the finite dimensional problem
\begin{equation}\label{eq:min_sigma_dis_restriction_para}
	(D_h^*\nabla \zeta_H,\nabla v_H)+(\sigma_H \zeta_H,v_H)=(f(T_1)-f(T_2)+\partial_{\tau} z^{\delta}(T_2)-\partial_{\tau} z^{\delta}(T_1),v_H),\quad\forall v_H\in V_{H}^0,
\end{equation}
where  $ \partial_{\tau}$ denote backward difference quotient of orde $k$ for some $0<\tau<\theta/k$.
Similar as Section \ref{sec:ellip-fem}, we use $H$ to denote the different spatial mesh size. Here  $\partial_{\tau}$ denotes the difference quotient as the discretized scheme \eqref{eq:min_D_dis_restriction_para} and $D_h^*$ is the diffusion coefficient we reconstructed in previous step with \textit{a priori} estimate
\begin{equation*}
	\|D_h^*-D^{\dag}\|_{L^2\II}\le c\delta^{\gamma}=:\epsilon \qquad  \text{with} ~~\gamma=\frac{ k}{4(1+\beta)(k+1)}.
\end{equation*}

The subsequent result offers an error estimation for $\sigma_H - \sigma^\dag$. Given that the proof parallels that of Theorem \ref{thm:error_weighted_potential}, we have decided not to reproduce it here.
\begin{theorem}\label{thm:error_weighted_potential_para}
	Suppose Assumption \ref{assume:coeff_positive_para} and \ref{assume:potential_regularity} holds, $D^{\dag}\in \mathcal{A}_D\cap W^{1,\infty}\II$  and $\|  D_h^*-D^{\dag}\|_{L^2\II}\le\epsilon$. Let $\zeta(\sigma^{\dag})$ be the solution to equation \eqref{eq:para_potential_0_bdry}, while $\sigma_H^*$ be the minimizer of \eqref{eq:min_sigma_dis_para}-\eqref{eq:min_sigma_dis_restriction_para}. Then with $\eta=H^2+\epsilon+(\tau^k + \delta\tau^{-1})^2+\sqrt{\alpha_2}$, there holds
	\begin{align*}
		\|  (\sigma^{\dag}-\sigma_H^*)\zeta(\sigma^{\dag})\|_{L^2\II}\le c \left(H\alpha_2^{-\frac{1}{2}}\eta+\eta+\tau^k +\delta\tau^{-1}+\left(\alpha_2^{-\frac{1}{2}}\eta(\min\{H+H^{-1}\eta,1\}+\epsilon)\right)^{\frac{1}{2}} \right).
	\end{align*}
Moreover, if $f_2-f_1\ge c>0$ a.e. in $\Omega$, $T_1 = T_0$ and $T_2 = 3T_0$ with $T_0$ being sufficiently large,
	then for any $\Omega' \Subset \Omega$, there exists a constant $c$ depending on $\mathrm{dist}(\Omega',\partial\Omega)$, $f$, $g$, $u_0$, $D^{\dag}$ and $\sigma^{\dag}$, such that
	\begin{align*}
		\|  (\sigma^{\dag}-\sigma_H^* \|_{L^2(\Omega')}\le c\left(H\alpha_2^{-\frac{1}{2}}\eta+\eta+\tau^k +\delta\tau^{-1}+\left(\alpha_2^{-\frac{1}{2}}\eta(\min\{H+H^{-1}\eta,1\}+\epsilon)\right)^{\frac{1}{2}} \right).
	\end{align*}
\end{theorem}

\begin{remark}\label{rem:pot-err}
	According to Theorem \ref{thm:error_weighted_diffusion} and Remark \ref{rem:error_diffusion_D}, we have
	$$ \| D_h^*  - D^\dag   \|  \le \epsilon =  c\, \delta^{\frac{k}{4(1+\beta)(k+1)}},$$
	provided that  $h \sim  \delta^{\frac{k}{2(k+1)}}$, $\alpha_1 \sim \delta^{\frac{2k}{k+1}}$ and the positivity condition \eqref{eq:positive_condition} is valid with $\beta\in[0,2]$.
	As a result, with the choice of parameters $H\sim \epsilon^{\frac{1}{2}},\alpha_2\sim \epsilon^2,\tau\sim\delta^{\frac{1}{k+1}}$, there holds the estimate
	\begin{equation*}
		\| (\sigma^{\dag}-\sigma_H^*)\zeta (\sigma^{\dag})\|_{L^2\II}\le c \,\delta^{\frac{k}{8(1+\beta)(k+1)}},
	\end{equation*}
	Finally, if $f_2-f_1\ge c>0$ a.e. in $\Omega$, $T_1 = T_0$ and $T_2 = 3T_0$ with $T_0$ being sufficiently large,
	then
	\begin{equation*}
		\| \sigma^{\dag}-\sigma_H^*\|_{L^2(\Omega')}\le c \,\delta^{\frac{k}{8(1+\beta)(k+1)}}.
	\end{equation*}
	where $\Omega' \Subset \Omega$ and  the constant $c$ depending on $\mathrm{dist}(\Omega',\partial\Omega)$, $f$, $g$, $u_0$, $D^{\dag}$ and $\sigma^{\dag}$.
\end{remark}

\section{Numerical results}\label{sec:numer}
In this section, we start with a concise introduction of our approach to addressing the optimization problems – \eqref{eq:min_D_dis}-\eqref{eq:min_D_dis_restriction}, \eqref{eq:min_D_dis_para}-\eqref{eq:min_D_dis_restriction_para}, \eqref{eq:min_sigma_dis}-\eqref{eq:min_sigma_dis_restriction}, and \eqref{eq:min_sigma_dis_para}-\eqref{eq:min_sigma_dis_restriction_para}. Following this, we present empirical results that demonstrate the precision of our proposed decoupled numerical algorithm.

\subsection{Numerical implementation}
In this part, we introduce the numerical implementation for the reconstruction of both the diffusion coefficient $D^\dag$ and potential $\sigma^\dag$.
 Our approach, inspired by our theoretical analysis, involves a sequential reconstruction of these two parameters. In particular, we initiate with the reconstruction of $D^\dag$, followed by the subsequent reconstruction of $\sigma^\dag$.
The noisy data, denoted as $z_i^{\delta}$ for $i=1,2$ in the elliptic problem, and $z^{\delta}$ for the parabolic problem, are respectively generated as follows:
\begin{align*}
	z_i^{\delta}(x)=u_i^{\dag}(x)+\delta\sup_{x\in\Omega}|u_i^{\dag}(x)|\xi(x),\quad\text{and}\quad z^{\delta}(x,t)=u^{\dag}(x,t)+\delta\sup_{x\in\Omega}|u^{\dag}(x,t)|\xi(x,t).
\end{align*}
Here, $\xi$ adheres to the standard Gaussian distribution, while $\delta$ denotes the level of noise. The solution $u^{\dag}$ corresponds to the precise values of $D^\dag$ and $\sigma^\dag$, calculated using a highly refined mesh. The conjugate gradient method is utilized to address the discrete optimization problem \cite[Section 5]{Nocedal:1999}. The derivative of the least-squares functional is computed and presented in the following lemma.

\begin{lemma}\label{lem:functional_derivative}
Consider the objective functional $J_{\alpha_1}$, which is defined in \eqref{eq:min_D_dis}-\eqref{eq:min_D_dis_restriction} or   \eqref{eq:min_D_dis_para}-\eqref{eq:min_D_dis_restriction_para}.
The derivative of $J_{\alpha_1}$ with respect to $q$, denoted as $J_{\alpha_1}'(q)$, for $q\in W^{1,2}\II$, is expressed as:
\begin{equation*}
J_{\alpha_1}'(q)=\nabla w\cdot\nabla v-\alpha_1\Delta q.
\end{equation*}
In the above equation, $v$ denotes a function which is the solution to the elliptic problem:
\begin{equation*}
\left\{ \begin{aligned}
-\nabla\cdot(q\nabla v) &= w^{\delta}-w, &&\mbox{\rm in } \Omega, \\
v&=0, &&\mbox{\rm on } \partial\Omega,
\end{aligned}
\right.
\end{equation*}
Similarly, consider another functional $J_{\alpha_2}$, defined in \eqref{eq:min_sigma_dis}-\eqref{eq:min_sigma_dis_restriction} or  \eqref{eq:min_sigma_dis_para}-\eqref{eq:min_sigma_dis_restriction_para}.
The derivative of $J_{\alpha_2}$ with respect to $\sigma$, denoted as $J_{\alpha_2}'(\sigma)$, for $\sigma\in W^{1,2}\II$, is given by:
\begin{equation*}
			 \quad J_{\alpha_2}'(\sigma)= \zeta \tilde{v}-\alpha_2\Delta \sigma,
\end{equation*}
where the function $\tilde{v}$ solves the elliptic problem
\begin{equation*}
\left\{\begin{aligned}
-\nabla\cdot(D_h^*\nabla \tilde{v})+\sigma\tilde{v} &= \zeta^{\delta}-\zeta, &&\mbox{\rm in } \Omega, \\
\tilde{v}&=0, &&\mbox{\rm on } \partial\Omega.
\end{aligned}\right.
\end{equation*}
\end{lemma}
\begin{proof}
To begin with, we calculate the derivative of $J_{\alpha_1}$ at $q$ in the direction of $p$, which is represented as $J'_{\alpha_1}(q)$.
A direct computation yields
\begin{equation*}
J'_{\alpha_1}(q)[p] = (w(q)-w^{\delta},w'(q)[p]) + \alpha_1 (\nabla q,\nabla p),
\end{equation*}
where $w'(q)[p] $ is the derivative of $w(q)$ at $q$ along the direction $p$. Note that $w'(q)[p]$ satisfies
\begin{align}\label{eqn:weak u_prime_q_p}
	(q\nabla w'(q)[p], \nabla\varphi)=-(p\nabla w(q), \nabla\varphi),\quad \forall \varphi \in { W_0^{1,2}(\Omega)}.
\end{align}
Meanwhile, the weak formulation of $v$ is given by
\begin{align}\label{eqn:weak_adjoint}
	(q\nabla v, \nabla\varphi) = (w^{\delta}-w(q),\varphi),\quad \forall \varphi\in { W_0^{1,2}(\Omega)}.
\end{align}
Now we choose $\varphi =v$ in \eqref{eqn:weak u_prime_q_p} and $\varphi=w'(q)[p]$ in \eqref{eqn:weak_adjoint}, and subtract the resulting identities. We obtain
$(w(q)-w^{\delta}), w'(q)[p]) =(p\nabla w,\nabla v)$. This shows the first assertion.

Moreover, the derivative  $J'_{\alpha_2}(\sigma)$ could be computed similarly.
Specifically, the derivative $J'_{\alpha_2}(\sigma)$ of $J_{\alpha_2}$ at $\sigma$ along the direction $p$ is given by
	\begin{equation*}
		J'_{\alpha_2}(\sigma)[p] = (\zeta(\sigma)-\zeta^{\delta},\zeta'(\sigma)[p]) + \alpha_2 (\nabla \sigma,\nabla p).
	\end{equation*}
	The derivative $\zeta'(\sigma)[p]$ of $\zeta(\sigma)$ at $\sigma$ along the direction $p$ satisfies
	\begin{align*}
		(D_h^*\nabla \zeta'(\sigma)[p], \nabla\varphi)+(\sigma\zeta'(\sigma)[p],\varphi)=-(p \zeta(\sigma), \varphi),\quad \forall \varphi \in { W_0^{1,2}(\Omega)}.
	\end{align*}
	Meanwhile, the weak formulation of $\tilde{v}$ is given by
	\begin{align*}
		(D_h^*\nabla \tilde{v}, \nabla\varphi)+(\sigma\tilde{v},\varphi) = (\zeta^{\delta}-\zeta(\sigma),\varphi),\quad \forall \varphi\in { W_0^{1,2}(\Omega)}.
	\end{align*}
	Choosing $\varphi =\tilde{v}$ in the former equation and $\varphi=\zeta'(\sigma)[p]$ in the later equation and subtracting, we obtain  $(\zeta(\sigma)-\zeta^{\delta}), \zeta'(\sigma)[p]) =(p\zeta,\tilde{v})$. The proof is completed.
\end{proof}

Lemma \ref{lem:functional_derivative} indicates that
the derivatives $J_{\alpha_1}'(q)$ and $J_{\alpha_2}'(\sigma)$ belongs to the dual space $(W^{1,2}(\Omega))'$ of $W^{1,2}(\Omega)$,
which is unsuitable for updating the coefficient $q$ and $\sigma$ directly. Therefore, we apply the Riesz map to pull them back to the space $W^{1,2}(\Omega)$ and obtain a sufficiently regular direction for updating. By Riesz representation theorem, there exists a unique function $G\in W^{1,2}(\Omega)$ such that
\begin{equation*}
	\langle J_{\alpha_1}'(q),\varphi\rangle_{(W^{1,2}(\Omega))',W^{1,2}(\Omega)}=(G,\varphi)_{W^{1,2}(\Omega)},\quad \forall \varphi\in { W^{1,2}(\Omega)},
\end{equation*}
where $	\langle \cdot,\cdot\rangle_{(W^{1,2}(\Omega))',W^{1,2}(\Omega)}$  denotes the duality pairing between $(W^{1,2}(\Omega))'$ and $W^{1,2}(\Omega)$, and $ (\cdot,\cdot)_{W^{1,2}(\Omega)}$ the $W^{1,2}(\Omega)$ inner product. Thus, $G\in W^{1,2}(\Omega)$ is the weak solution of the following elliptic problem:
\begin{equation*}
	\left\{
	\begin{aligned}
		-\Delta G+ G&= J_{\alpha_1}'(q),\quad \mbox{in }\Omega,\\
		\partial_{\nu} G&=0, \quad \mbox{on }\partial\Omega.
	\end{aligned}
	\right.
\end{equation*}
Then the function $G$ serves as the gradient of the functional $J_{\alpha_1}$ at $q$, and is used in practical computation. The gradient of the functional  $J_{\alpha_2}$ at $\sigma$ can be computed in the same formulation.

\subsection{Numerical experiments}
In this part, we present numerical results for reconstructing diffusion coefficient $D$ and potential $\sigma$.
 The reconstruction accuracy is measured in relatively $L^2(\Omega)$ error:
\begin{align*}
	e_D=\|  D_h^*-D^{\dag}\|_{L^2\II}/\|D^{\dag}\|_{L^2(\Omega)}\quad\text{and}\quad e_{\sigma}=\|  \sigma_h^*-\sigma^{\dag}\|_{L^2\II}/\|\sigma^{\dag}\|_{L^2(\Omega)}.
\end{align*}

To begin with, we present numerical results for one- and two-dimensional elliptic equations.
\vskip5pt

\begin{example}[1-D elliptic equation]\label{ex:ellptic_1D_smooth}
$\Omega=(0,1)$, $D^{\dag}(x)=2+\sin(2\pi x)$, $\sigma^{\dag}(x)=1+x(1-x)$.
The boundary is $g\equiv 1$ and the two sources are given by $f_1\equiv 1$ and $f_2\equiv 10$.
The exact solution $u_i^{\dag}$ are computed with mesh size $h=1/1600$.
\end{example}\vskip5pt
The results of the reconstruction at different noise levels can be seen in Figure \ref{Fig:elliptic_1D_sm},
while the relative errors are displayed in Table \ref{tab:ex_elliptic}.
For different noise level $\delta$, we adopt the regularization parameter $\alpha_1$ and mesh size $h$ as $\alpha_1=C_{\alpha_1}\delta^2$ and $h=C_h\delta^{\frac{1}{2}}$ respectively. This choice is guided by the recommendations made in Remark \ref{rem:error_diffusion_D} for the reconstruction of $D^\dag$. Next, in the process of reconstructing $\sigma^\dag$, we follow the guidelines provided in Remark \ref{rem:error_potential_sigma}. Specifically, we assign values to $\alpha_2$ and $H$ as $\alpha_2= C_{\alpha_2}\delta^{2\gamma}$ and $H= C_H\delta^{\frac{\gamma}{2}}$ respectively. Here, $\gamma$ represents the empirical convergence rates observed in our experiments.
The constant $C_{\alpha_1}$, $C_h$, $ C_{\alpha_2}$ and $ C_H$  are determined by a trial and error way.
For reconstruction of $D^\dag$, we initially take $\alpha_1=10^{-6}$ and $h=1/16$. The numerical results indicate that the error $e_D$ decays to zero as the noise level tends to zero, with rate $O(\delta^{0.52})$. For reconstruction of potential $\sigma^\dag$, we initially take $\alpha_2=10^{-5}$ and $H=1/16$ and observe a convergence rate $O(\delta^{0.18})$.
It's important to note, as discussed in Remark 3.1, that the predicted rate for $D^\dag$ is $O(\delta^{1/4})$, which is significantly lower than the empirically observed rate. Moreover, with the empirical rate $\gamma = 0.52$, the predicted rate for $\sigma^\dag$ is expected to be $O(\delta^{0.26})$ as noted in Remark 3.2.
However, this rate is seldom observed in practical applications.
\begin{figure}[htbp]
	\centering
	\begin{tabular}{ccc}
		\includegraphics[width=0.27\textwidth]{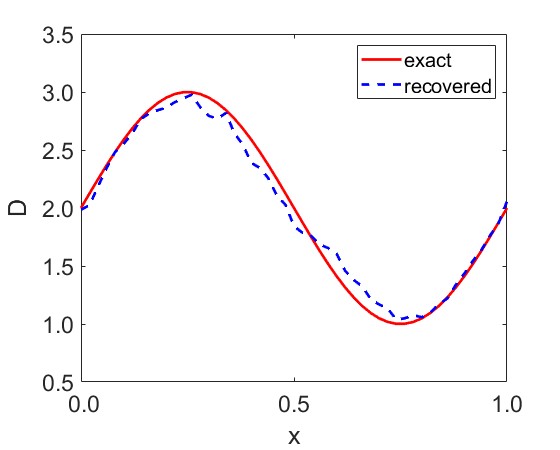}&
		\includegraphics[width=0.27\textwidth]{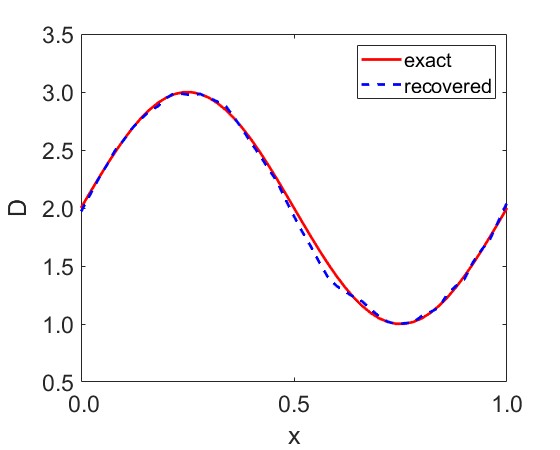}&
		\includegraphics[width=0.27\textwidth]{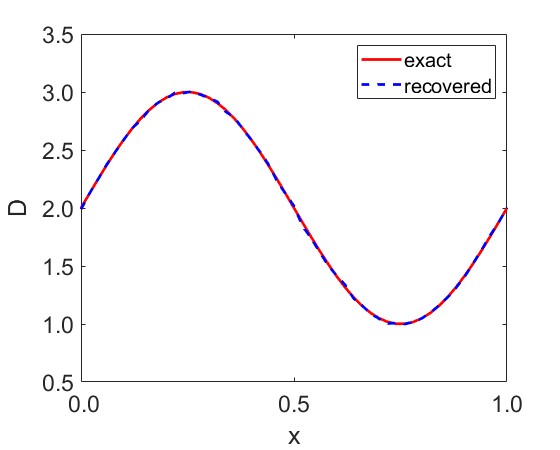}\\
		\includegraphics[width=0.27\textwidth]{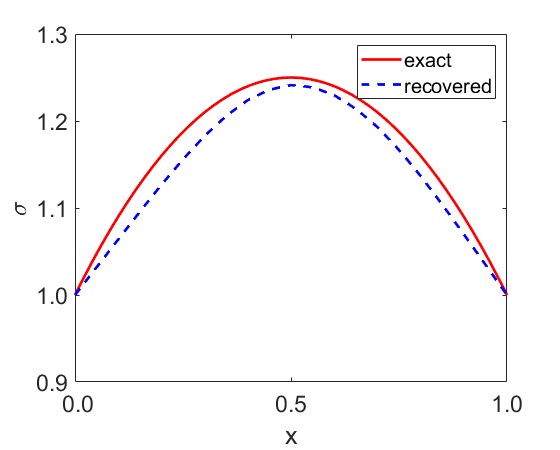}&
		\includegraphics[width=0.27\textwidth]{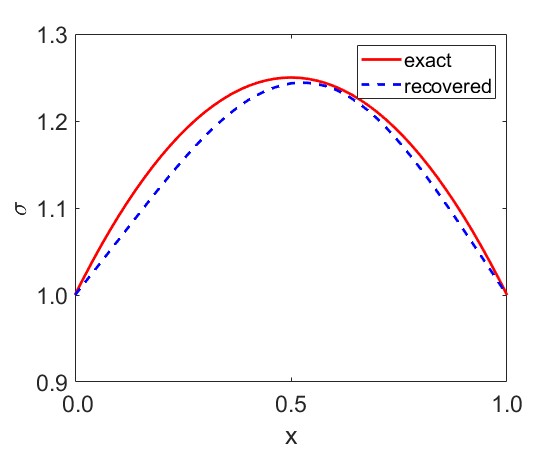}&
		\includegraphics[width=0.27\textwidth]{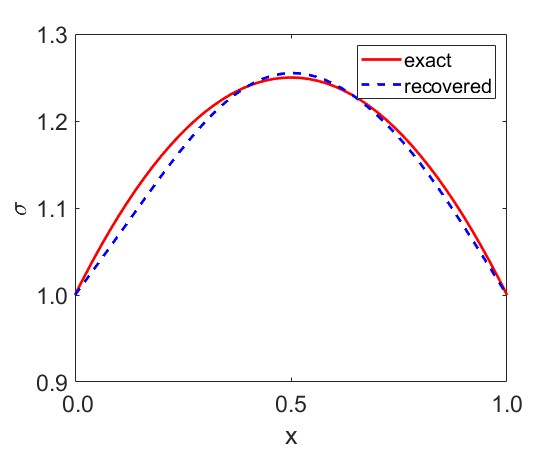}\\
		(a) $\delta=10^{-2}$ & (b) $\delta=10^{-3}$ & (c) $\delta=10^{-4}$
	\end{tabular}
	\caption{Example \ref{ex:ellptic_1D_smooth}. 
	First row: reconstructions of $D^\dag$. Second row: reconstructions of $\sigma^\dag$.}
	\label{Fig:elliptic_1D_sm}
\end{figure}

 \begin{table}[htp!]
	\centering
	\caption{Examples \ref{ex:ellptic_1D_smooth} and \ref{ex:ellptic_2D_smooth}: convergence with respect to $\delta$.\label{tab:ex_elliptic}} 
	\begin{tabular}{c|ccccc|ccccc|}
		\toprule
		\multicolumn{1}{c}{}&
		\multicolumn{5}{c}{(a) Example \ref{ex:ellptic_1D_smooth}}&\multicolumn{5}{c}{(b) Example \ref{ex:ellptic_2D_smooth}}\\
		\cmidrule(lr){2-6} \cmidrule(lr){7-11}
		$\delta$   & 1e-2 & 5e-3 & 1e-3 & 5e-4 &1e-4   & 10e-2 & 5e-2 & 1e-2 & 5e-3 &1e-3\\
		\midrule
		$e_D$ &4.87e-2 & 3.51e-2 & 1.73e-2 & 7.12e-3 & 4.67e-3 & 1.29e-1 & 6.34e-2 & 3.20e-2 & 1.95e-2 & 1.14e-2  \\
		$e_{\sigma}$ & 1.78e-2  & 1.73e-2 & 1.61e-2 & 1.00e-2 & 9.24e-3 &7.70e-2 & 3.55e-2 & 3.12e-2 & 2.23e-2 & 2.13e-2\\
		\bottomrule
	\end{tabular}
\end{table}

\begin{example}[2-D elliptic equation]\label{ex:ellptic_2D_smooth}
$\Omega=(0,1)^2$,  $D^{\dag}(x,y)=2+\sin(2\pi x)\sin(2\pi y)$ and $\sigma^{\dag}(x,y)=1+y(1-y)\sin(\pi x)$. The boundary is $g\equiv 1$ and the two sources are given by $f_1\equiv 1$, $f_2\equiv 10$. The exact solution $u_i^{\dag}$ are computed  with mesh size $h=1/200$.
\end{example}

The numerical results for Example \ref{ex:ellptic_2D_smooth}
are presented in Table \ref{tab:ex_elliptic} and Fig. \ref{Fig:elliptic_2D_sm}.
The regularization parameters  and mesh sizes are initialized to $h=1/16$,
$\alpha_1=10^{-8}$, $h=1/16$,  $\alpha_2=5\times 10^{-6}$  and $H=1/12 $
The empirical convergence rates for $e_D$ and $e_\sigma$ with respect to $\delta$
are about $O(\delta^{0.51})$ and $O(\delta^{0.25})$, respectively, which are
comparable with that for Example \ref{ex:ellptic_1D_smooth}

%
\begin{figure}[htbp]
	\centering
	\begin{tabular}{ccc}
		\includegraphics[width=0.27\textwidth]{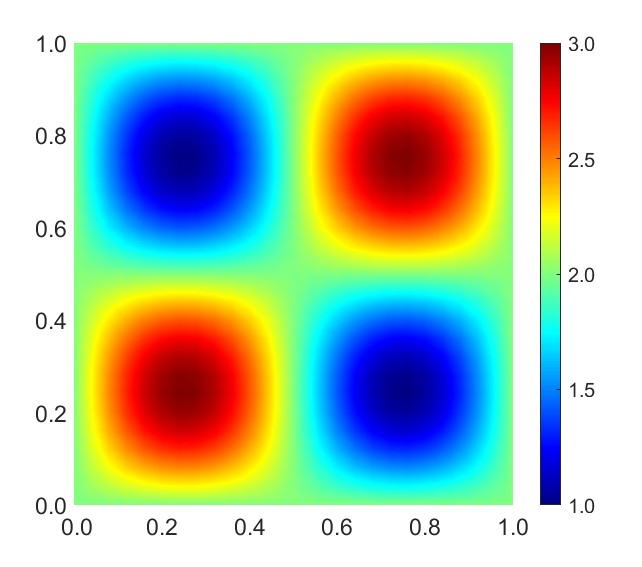}&
		\includegraphics[width=0.27\textwidth]{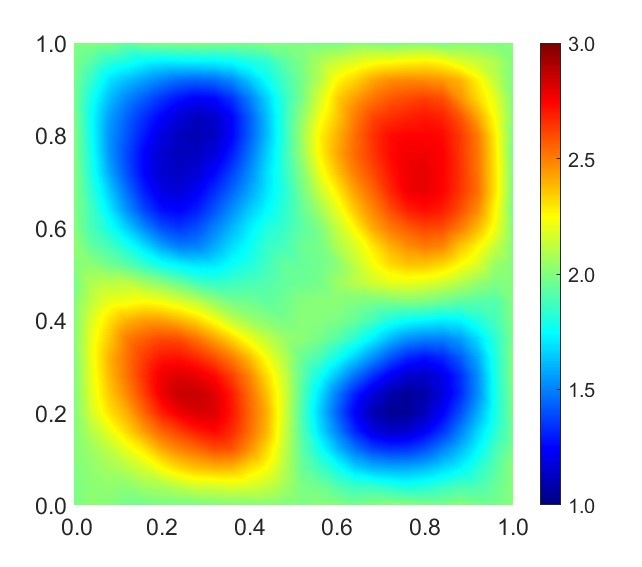}&
		\includegraphics[width=0.27\textwidth]{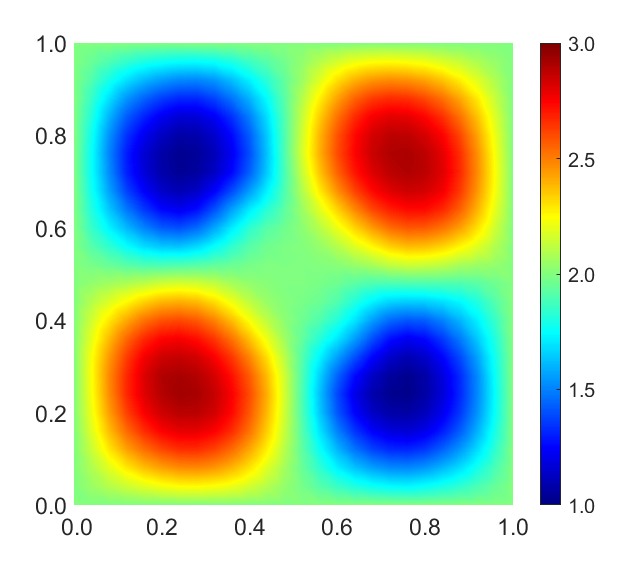}\\
		\includegraphics[width=0.27\textwidth]{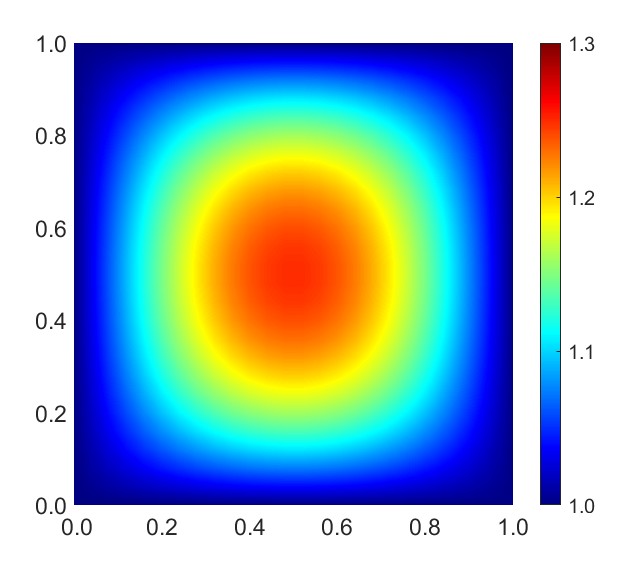}&
		\includegraphics[width=0.27\textwidth]{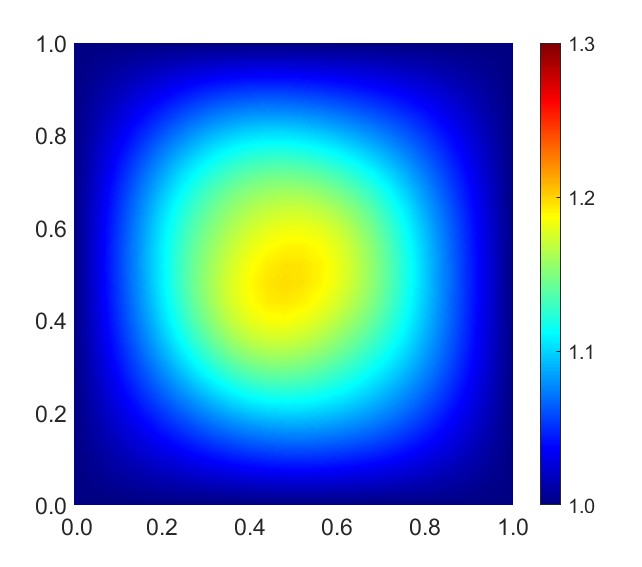}&
		\includegraphics[width=0.27\textwidth]{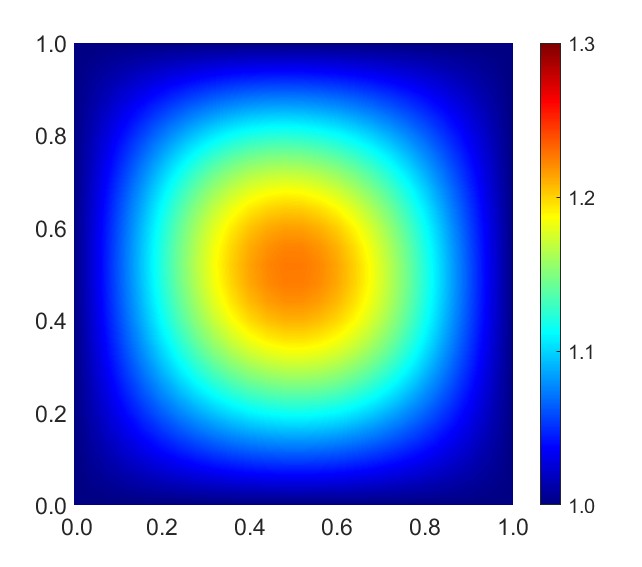}\\
		(a) exact & (b) $\delta=5\times10^{-3}$ & (c) $\delta=5\times10^{-4}$
	\end{tabular}
	\caption{Example \ref{ex:ellptic_2D_smooth}. First row: reconstructions of $D^\dag$.
		Second row: reconstructions of $\sigma^\dag$. }
	\label{Fig:elliptic_2D_sm}
\end{figure}

\begin{figure}[!h]
	\centering
	\subfigure[$e_D$]{
		\includegraphics[width=0.28\textwidth]{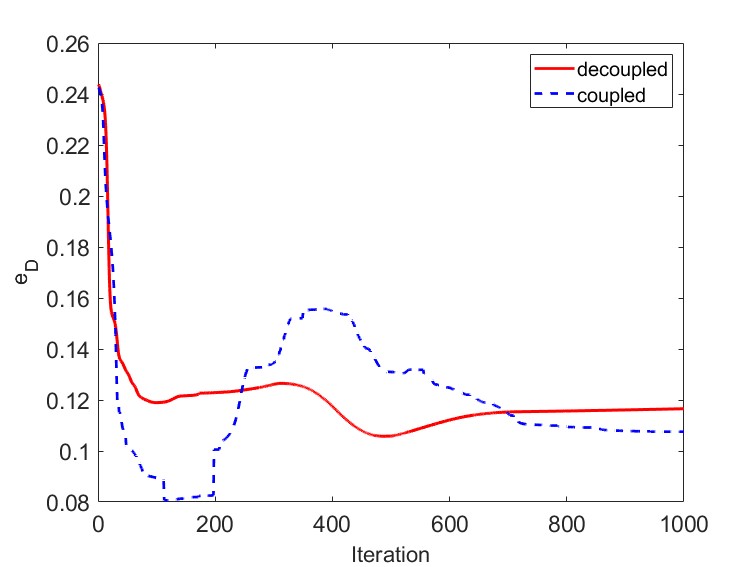}}
	\subfigure[$D_h^*$ coupled scheme]{
		\includegraphics[width=0.25\textwidth]{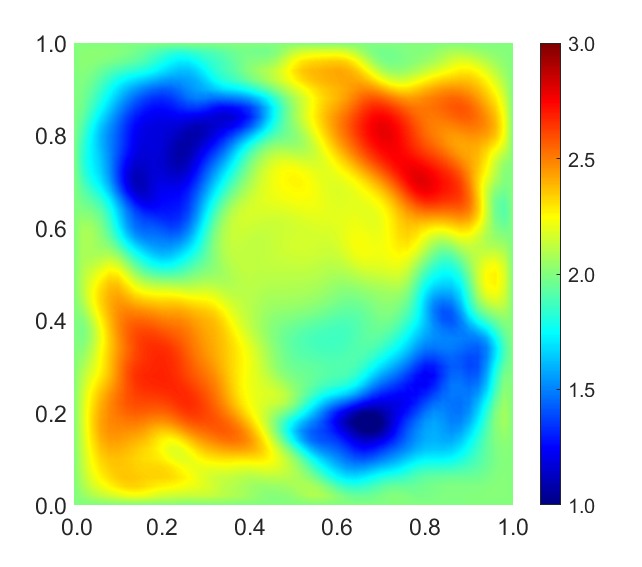}}
	\subfigure[$D_h^*$ decoupled scheme]{
		\includegraphics[width=0.25\textwidth]{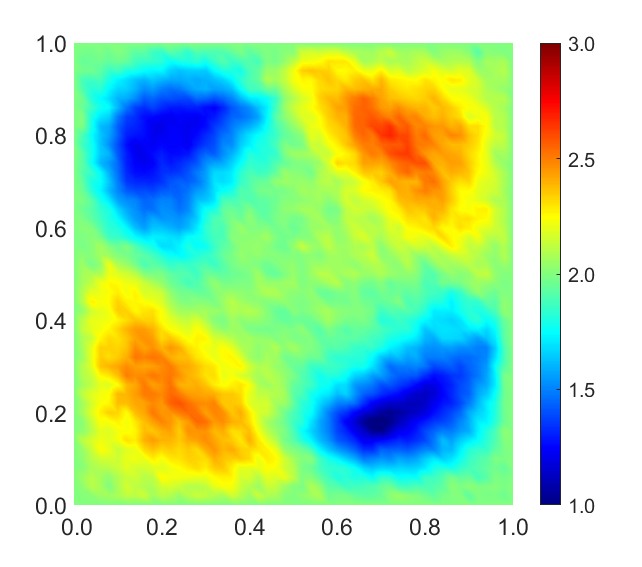}}

	\subfigure[$e_\sigma$]{
		\includegraphics[width=0.28\textwidth]{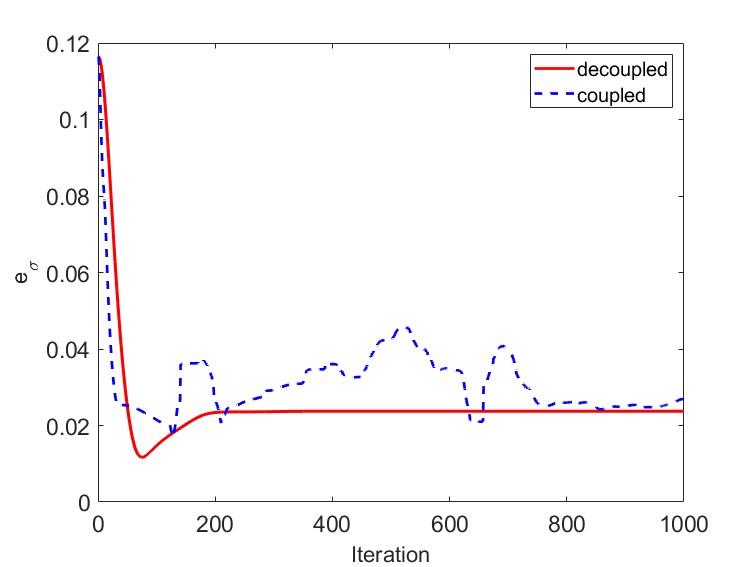}}
	\subfigure[$\sigma_h^*$ coupled scheme]{
		\includegraphics[width=0.25\textwidth]{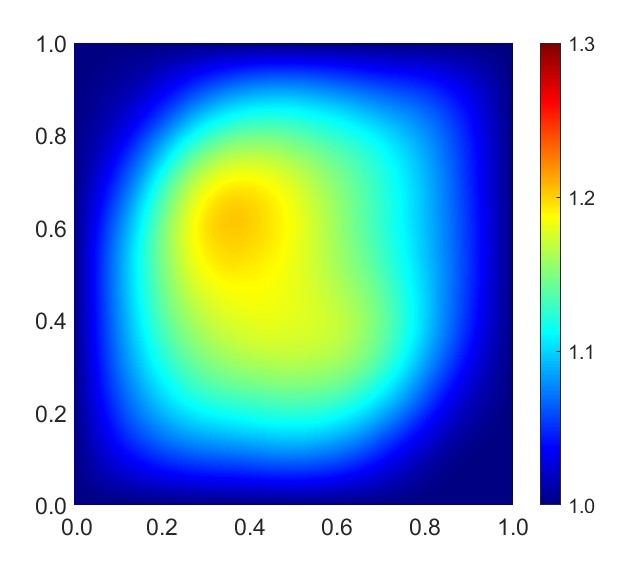}}
	\subfigure[$\sigma_h^*$ decoupled scheme]{
		\includegraphics[width=0.25\textwidth]{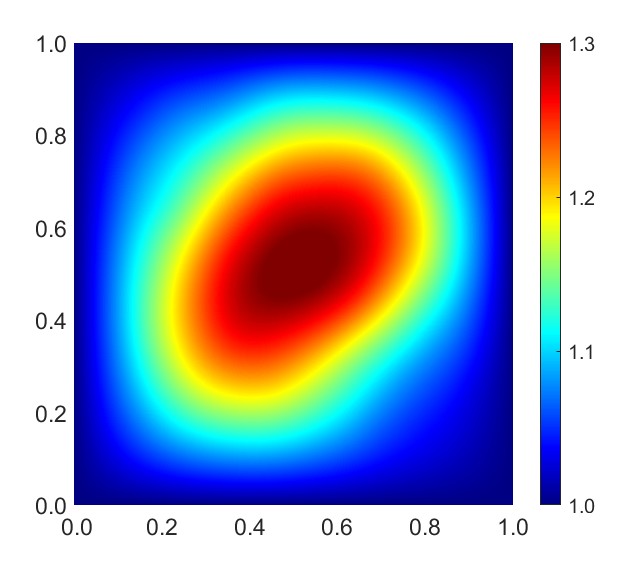}}
	\caption{Comparison between the proposed decoupled algorithm and the coupled scheme \eqref{eqn:coupled_reconstruction}-\eqref{eqn:coupled_reconstruction}.}
	\label{Fig:comparison}
\end{figure}
    
Next, we compare our decoupled reconstruction process with
the scheme \eqref{eqn:coupled_reconstruction}-\eqref{eqn:coupled_eq} where we compute $D_h^*$ and $\sigma_h^*$ simultaneously. We address the optimization problem defined in  \eqref{eqn:coupled_reconstruction}-\eqref{eqn:coupled_eq}, using the conjugate gradient method. The solution process alternates between two directions. We initially set $\sigma_h$ as fixed and employ the conjugate gradient descent to optimize $D_h$. Subsequently, we fix $D_h$ and utilize the conjugate gradient descent to optimize $\sigma_h$. We continue this alternating process until convergence is reached.
In the computation, we select a mesh size of $h=1/50$. The regularization parameters are determined through a process of trial and error.
The reconstruction results, carried out with a noise level of $\delta=2\%$, are presented in Figure \ref{Fig:comparison}. The first column illustrates the convergence of the conjugate gradient iteration. Notably, the errors for the decoupled scheme demonstrate a rapid and steady decay.
However, for the coupled scheme as defined in \eqref{eqn:coupled_reconstruction}-\eqref{eqn:coupled_eq}, the optimization problem is considerably more complex. As a result, the errors exhibit a period of oscillation and require a significantly longer time to converge.

Next, we present numerical results for the parabolic equation.\vskip5pt

\begin{example}[1-D parabolic equation]\label{ex:para_1D_nonsmooth}
 $\Omega=(0,1)$,  $D^{\dag}(x)=2+\sin(2\pi x)$, $\sigma^{\dag}(x)=1-|x-\frac{1}{2}|^{1.1}$, $u_0(x)=1+\frac{1}{2}\sin(\pi x)$ and $g\equiv 1$.
 We take the source term as
 \begin{equation}\label{eqn:para-f}
f(x,t) =\left\{
\begin{aligned}
		&1,\quad \text{for } t\in[0,1.5];\\
		&\tfrac{9}{2}\sin (\tfrac{\pi}{2}(t-\tfrac{5}{2}))+\tfrac{11}{2},\quad\text{for } t\in (1.5,3.5);\\
		&10,\quad \text{for }t\in[3.5,\infty).
\end{aligned}
\right.
\end{equation}
 The measurement is taken in the time-space domain
 $(t,x)\in [0.9,1]\times\Omega$ and $(t,x)\in [4.9,5]\times\Omega$.
The exact solution $u^{\dag}$ are generated with mesh size $h=1/1600$ and $\tau=1/2000$.
Throughout, we use backward Euler scheme, i.e. $k=1$, to discretize in time variable.
\end{example}\vskip5pt

The reconstruction results are listed in Table \ref{tab:ex_para} and Figure \ref{Fig:para_1D_nsm}. For reconstructing diffusion coefficient $D^\dag$, the parameters are taken to be
$h=C_h\delta^{\frac{1}{4}}$, $\tau=C_{\tau}\delta^{\frac{1}{2}}$ and $\alpha_1=C_{\alpha_1}\delta$, according to Theorem \ref{thm:error_weighted_diffusion_para}.
We observe $e_D$ decays in a rate $O(\delta^{0.49})$ with initial mesh size $h=1/16$, time step $\tau=0.1$ and  regularization parameter $\alpha_1=10^{-6}$. As shown in Theorem \ref{thm:error_weighted_potential_para}, for reconstructing $\sigma^\gamma$, we take parameters $H= C_H\delta^{\frac{\gamma}{2}}$, $\tau= C_{\tau}\delta^{\frac{1}{2}}$ and $\alpha_2= C_{\alpha_2}\delta^{2\gamma}$. Here, $\gamma$ represents the empirical convergence rates for recovering $D^\dag$ observed in our experiments.
We initialize the mesh size $H=1/16$, time step $\tau=0.1$ and  regularization parameter $\alpha_2=10^{-5}$. The numerical results show that the error  $e_{\sigma}$ have decay rates  $O(\delta^{0.22})$.
These results are  comparable with that for Examples \ref{ex:ellptic_1D_smooth}  and \ref{ex:ellptic_2D_smooth}.

 \begin{figure}[!h]
	\centering
	\begin{tabular}{ccc}
		\includegraphics[width=0.27\textwidth]{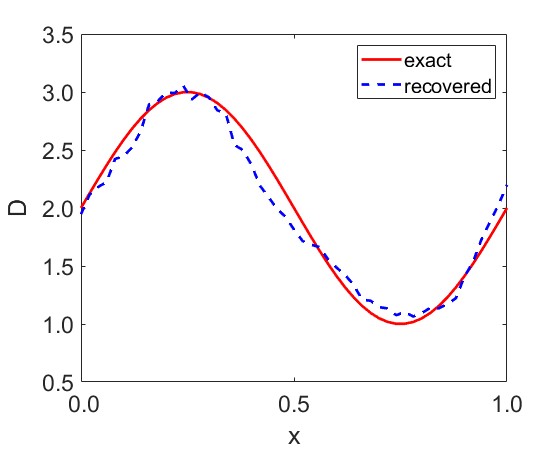}&
		\includegraphics[width=0.27\textwidth]{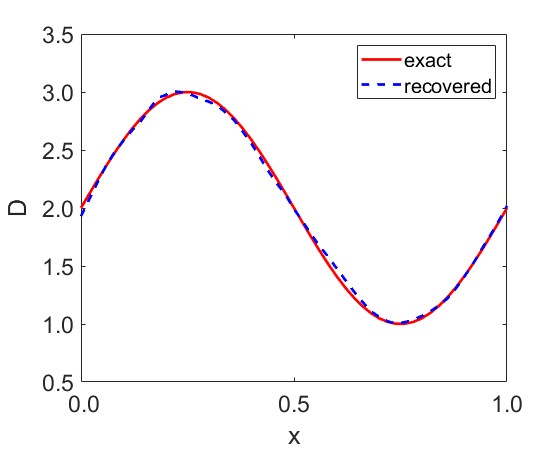}&
		\includegraphics[width=0.27\textwidth]{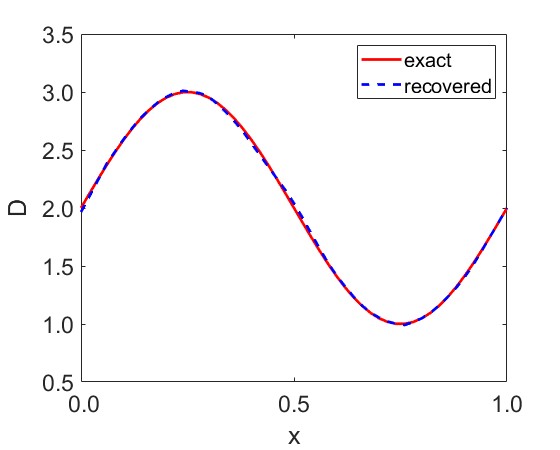}\\
		\includegraphics[width=0.27\textwidth]{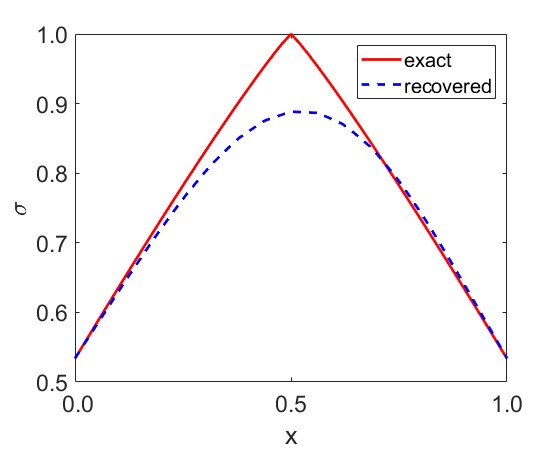}&
		\includegraphics[width=0.27\textwidth]{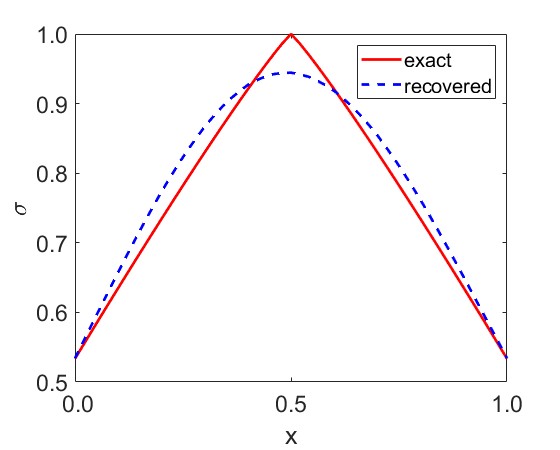}&
		\includegraphics[width=0.27\textwidth]{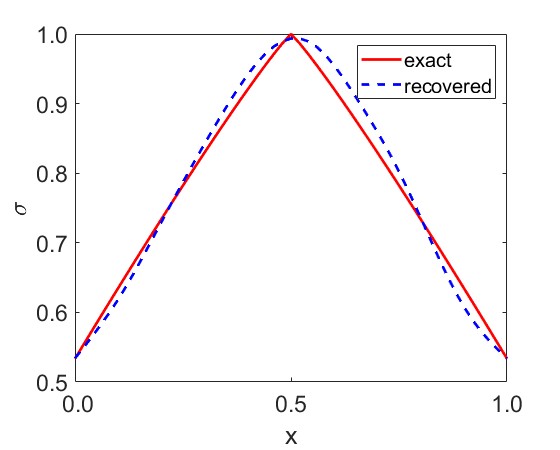}\\
		(a) $\delta=10^{-2}$ & (b) $\delta=10^{-3}$ & (c) $\delta=10^{-4}$
	\end{tabular}
	\caption{Example \ref{ex:para_1D_nonsmooth}. First row: reconstructions of $D^\dag$.
		Second row: reconstructions of $\sigma^\dag$.}
	\label{Fig:para_1D_nsm}
\end{figure}

\begin{table}[htp!]
	\centering
	\caption{Examples \ref{ex:para_1D_nonsmooth} and \ref{ex:para_2D_smooth}:
	convergence with respect to $\delta$.\label{tab:ex_para}} 
	\begin{tabular}{c|ccccc|ccccc|}
		\toprule
		\multicolumn{1}{c}{}&
		\multicolumn{5}{c}{(a) Example \ref{ex:para_1D_nonsmooth}}&\multicolumn{5}{c}{(b) Example \ref{ex:para_2D_smooth}}\\
		\cmidrule(lr){2-6} \cmidrule(lr){7-11}
		$\delta$   & 1e-2 & 5e-3 & 1e-3 & 5e-4 &1e-4   & 1e-2 & 5e-3 & 1e-3 & 5e-4 &1e-4\\
		\midrule
		$e_D$ & 5.49e-2 & 3.78e-2 & 1.37e-2 & 7.45e-3 & 6.69e-3 & 1.59e-1 & 1.00e-1 & 3.15e-2 & 1.98e-3 & 8.78e-3\\
		$e_{\sigma}$ & 5.23e-2  & 5.14e-2 & 3.32e-2 & 2.29e-2 & 2.13e-2 & 4.41e-2 &  2.43e-2 & 2.24e-2 & 1.98e-2 & 1.53e-2\\
		\bottomrule
	\end{tabular}
\end{table}

\begin{example}[2-D parabolic equation]\label{ex:para_2D_smooth}
$\Omega=(0,1)^2$, $D^{\dag}(x,y)=2+e^{-20(x-0.5)^2-20(y-0.7)^2}-e^{-20(x-0.5)^2-20(y-0.3)^2}$, $\sigma^{\dag}(x,y)=1+0.5e^{-20(x-0.6)^2-20(y-0.6)^2}$,  $u_0(x)=1+\frac{1}{2}\sin(\pi x)\sin(\pi y)$ and $g\equiv 1$. The source term is given by \eqref{eqn:para-f}. The measurement data is observed  in the time-space domain $(t,x)\in [0.9,1]\times\Omega$ and $(t,x)\in [4.9,5]\times\Omega$.
 The exact solution $u^{\dag}$ are generated with mesh size $h=1/200$ and $\tau=1/250$.
\end{example}

The numerical results for Example \ref{ex:para_2D_smooth} are shown in Table \ref{tab:ex_para} and Figure \ref{Fig:para_2D_sm}. The process of obtaining these results was split into two steps.
In the first step, the diffusion coefficient was recovered. The parameters for this step were chosen based on the scaling provided in Remark \ref{rem:error_diffusion_D_para}. The initializations for these parameters were set as follows: $h=1/16$, $\tau=0.1$, and $\alpha_1=10^{-6}$.
In the second step, the potential was reconstructed. The parameters for this step were selected according to the guidelines in Remark \ref{rem:pot-err}, and were initialized to $H=1/16 $, $\tau=0.1$, and
$\alpha_2=10^{-6}$. It was discovered that the empirical convergence rate for $e_D$ in relation to $\delta$ was approximately $O(\delta^{0.48})$. This rate is higher than the theoretical one. Moreover, the empirical rate for $e_\sigma$ was found to be about $O(\delta^{0.22})$, which aligns with the theoretical estimate.

 \begin{figure}[!h]
	\centering
	\begin{tabular}{ccc}
		\includegraphics[width=0.27\textwidth]{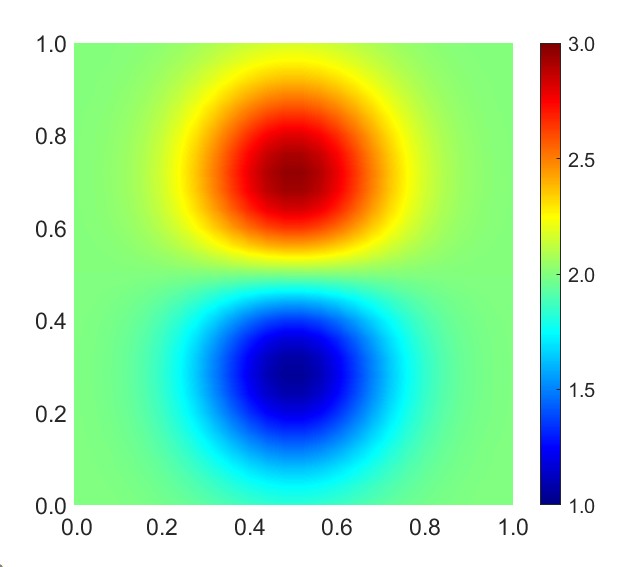}&
		\includegraphics[width=0.27\textwidth]{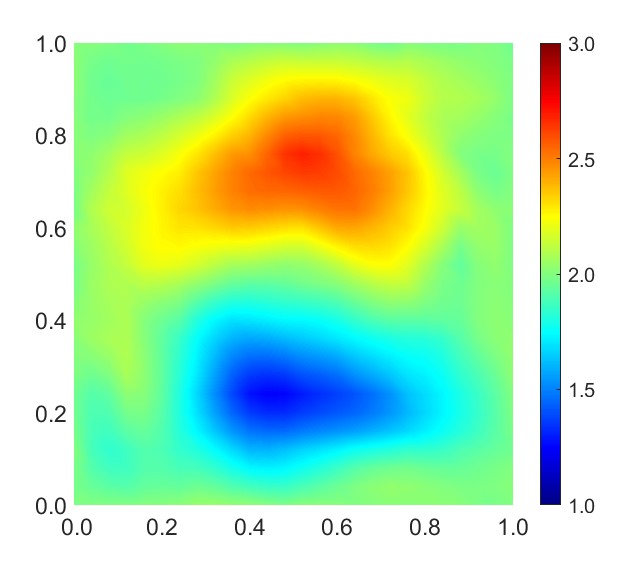}&
		\includegraphics[width=0.27\textwidth]{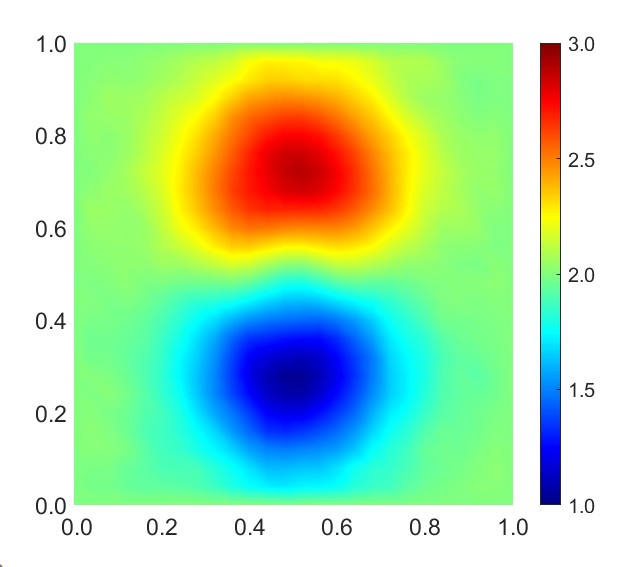}\\
		\includegraphics[width=0.27\textwidth]{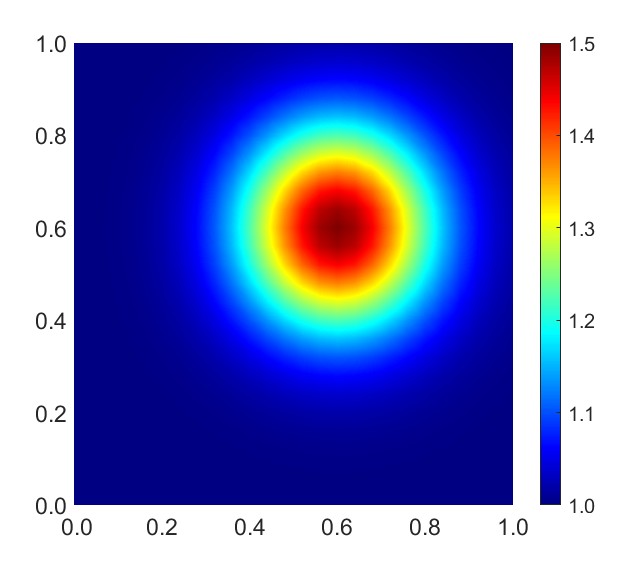}&
		\includegraphics[width=0.27\textwidth]{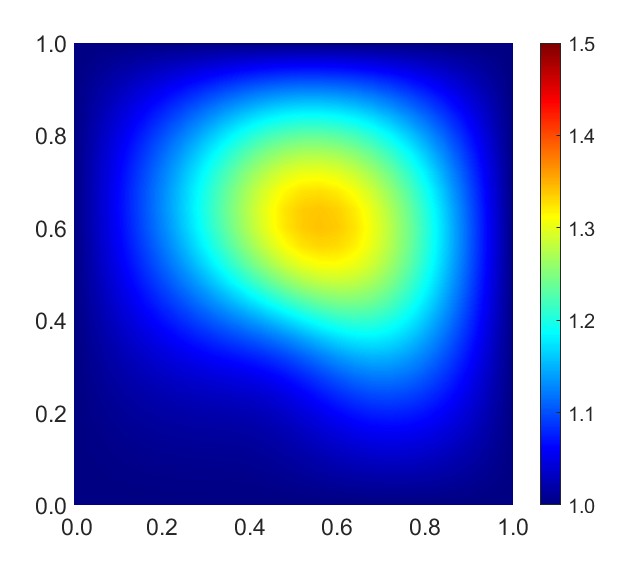}&
		\includegraphics[width=0.27\textwidth]{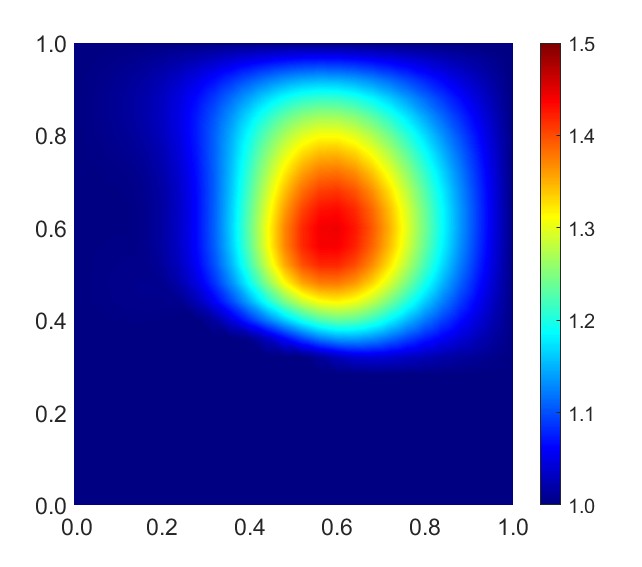}\\
		(a) exact & (b) $\delta=5\times 10^{-3}$ & (c) $\delta=5\times 10^{-4}$
	\end{tabular}
	\caption{Example \ref{ex:para_2D_smooth}. First row: reconstructions of $D^\dag$.
		Second row: reconstructions of $\sigma^\dag$.}
	\label{Fig:para_2D_sm}
\end{figure}

\section{Conclusion}
In this paper, we investigated the simultaneous reconstruction of the diffusion and potential coefficients inherent in elliptic/parabolic equations.  This is achieved through the utilization of two internal measurements of the solutions. We proposed a decoupled algorithm capable of sequentially recovering these two parameters. The approach begins with a straightforward reformulation leading to a standard problem of identifying the diffusion coefficient. This coefficient is numerically recovered, without any requirement for knowledge of the potential, by employing an output least-square method in conjunction with finite element discretization. In the subsequent step, the previously recovered diffusion coefficient becomes instrumental in the reconstruction of the potential coefficient.  The approach is inspired by a constructive conditional stability, and we have provided rigorous \textsl{a priori} error estimates in $L^2(\Omega)$ for the recovered diffusion and potential coefficients. The derivation of these estimates necessitated the development of a weighted energy argument and suitable positivity conditions. These estimates serve as a helpful guide to choose appropriate regularization parameters and discretization mesh sizes, aligned with the noise level.

Interestingly, our numerical experiments indicated that the empirical rates surpassed the theoretical ones. Future work will focus on investigating the reasons behind this discrepancy and improving the error estimate. The scope of this paper was confined to cases where observational data was taken over the entire domain $\Omega$. It would be intriguing to extend the argument to the subdomain observation or the boundary observation. In regard to the parabolic case, future research will consider scenarios where observations are taken at fixed time points $T_1$ and $T_2$, instead of across the time-space domain. Finally, we intend to examine the numerical analysis of the hybrid inverse problem, where the observation is $\sigma u_i$. This could be achieved by exploring a  different conditional stability that complies with a provable positivity condition. The current work forms a solid foundation for these future investigations.

\bibliographystyle{abbrv}

\end{document}